\documentclass[12pt]{article}
\usepackage{amsmath,bm,amssymb}
\usepackage{amsthm}
\usepackage{graphicx,psfrag,epsf}
\usepackage{enumerate}
\usepackage{natbib}
\usepackage{url} 
\usepackage{amsfonts}
\usepackage[colorlinks,linkcolor=red,anchorcolor=blue, citecolor=blue]{hyperref}
\usepackage{float}
\usepackage[T1]{fontenc}
\usepackage[utf8]{inputenc}
\usepackage{authblk}
\usepackage{epstopdf}
\usepackage{subfigure}

\usepackage[toc,page]{appendix}
\newcommand{\blind}{0}
\usepackage{threeparttable}
\numberwithin{equation}{section}

\newtheorem{theorem}{Theorem}

\newtheorem{lemma}{Lemma}

\newtheorem{condition}{Condition}

\theoremstyle{definition}

\newtheorem{corollary}{Corollary}

\newtheorem{eg}{Example}
\theoremstyle{remark}

\allowdisplaybreaks

\begin{document}

\def\spacingset#1{\renewcommand{\baselinestretch}%
{#1}\small\normalsize} \spacingset{1}


\if0\blind
{
  \title{\bf  On Prediction Properties of Kriging: Uniform Error Bounds and Robustness}
   \author[1]{Wenjia Wang\thanks{Wenjia Wang is a postdoctoral fellow in the Statistical and Applied Mathematical Sciences Institute, Durham, NC 27709, USA (Email: wenjia.wang234@duke.edu);}}
   \author[2]{Rui Tuo\thanks{Rui Tuo is Assistant Professor in Department of Industrial and Systems Engineering, Texas A\&M University, College Station, TX 77843, USA (Email: ruituo@tamu.edu); Tuo's work is supported by NSF grant DMS 1564438 and NSFC grants 11501551, 11271355 and 11671386. }}
\author[3]{C. F. Jeff Wu\thanks{C. F. Jeff Wu is Professor and Coca-Cola Chair in Engineering Statistics at School of Industrial and Systems Engineering, Georgia Institute of Technology, Atlanta, GA 30332, USA (Email: jeff.wu@isye.gatech.edu). Wu's work is supported by NSF grant DMS 1564438.\\ \indent The authors are grateful to an AE and referees for very helpful comments. }}
\affil[1]{The Statistical and Applied Mathematical Sciences Institute, Durham, NC 27709, USA}
\affil[2]{Department of Industrial and Systems Engineering, Texas A\&M University, College Station, TX 77843, USA}
\affil[3]{The H. Milton Stewart School of Industrial and Systems Engineering, Georgia Institute of Technology, Atlanta, GA 30332, USA}

\renewcommand\Authands{ and }
    \date{}
  \maketitle

} \fi
\date{}
\if1\blind
{
  \bigskip
  \bigskip
  \bigskip
  \begin{center}
    {\LARGE\bf  On Prediction Properties of Kriging: Uniform Error Bounds and Robustness}
\end{center}
  \medskip
} \fi

\bigskip
\begin{abstract}
Kriging based on Gaussian random fields is widely used in reconstructing unknown functions.
The kriging method has pointwise predictive distributions which are computationally simple.
However, in many applications one would like to predict for a range of untried points simultaneously. In this work we obtain some error bounds for the simple and universal kriging predictor under the uniform metric. It works for a scattered set of input points in an arbitrary dimension, and also covers the case where the covariance function of the Gaussian process is misspecified. These results lead to a better understanding of the rate of convergence of kriging under the Gaussian or the Mat\'ern correlation functions, the relationship between space-filling designs and kriging models, and the robustness of the Mat\'ern correlation functions.
\end{abstract}

\noindent%
\vfill

\newpage
\spacingset{1.45} 
\section{Introduction}
\label{sec:intro}

Kriging is a widely used methodology to reconstruct functions based on their scattered evaluations. Originally, kriging was introduced to geostatistics by \cite{matheron1963principles}. Later, it has been applied to computer experiments \citep{sacks1989design}, machine learning \citep{rasmussen2006gaussian}, small area estimation from survey data \citep{rao2015small}, and other areas.
With kriging, one can obtain an interpolant of the observed data, that is, the predictive curve or surface goes through all data points. Conventional regression methods, like the linear regression, the local polynomial regression \citep{fan1996local} and the smoothing splines \citep{wahba1990spline}, do not have this property. It is suitable to use interpolation  in spatial statistics and machine learning when the random noise of the data is negligible. The interpolation property is particularly helpful in computer experiments, in which the aim is to construct a surrogate model for a deterministic computer code, such as a finite element solver.

A key element of kriging prediction is the use of conditional inference based on Gaussian processes. At each untried point of the design region (i.e., domain for the input variables), the conditional distribution of a Gaussian process is normal with explicit mean and variance.
The pointwise confidence interval of the kriging predictor is then constructed using this conditional distribution.
In many applications, it is desirable to have a joint confidence region of the kriging predictor over a continuous set of the input variables such as an interval or rectangular region. The pointwise confidence interval for each design point cannot be amalgamated over the points in the design region to give a confidence region/limit with guaranteed coverage probability, even asymptotically. To address this question, it would be desirable to have a theory that gives good bounds on the worst (i.e., maximum) error of the kriging predictor over the design region. This bound can be useful in the construction of confidence regions with guaranteed coverage property, albeit somewhat conservatively.

In this work, we derive error bounds of the simple and universal kriging predictor under a uniform metric. The predictive error is bounded in terms of the maximum pointwise predictive variance of kriging.
A key implication of our work is to show that the overall predictive performance of a Gaussian process model is tied to the smoothness of the underlying correlation function as well as the space-filling property of the design (i.e., collection of the design points).
This has two major consequences. First, we show that a less smooth correlation function is more \textit{robust} in prediction, in the sense that prediction consistency can be achieved for a broader range of true correlation functions, while a smoother correlation function can achieve a higher rate of convergence provided that it is no smoother than the true correlation. 
Second, these error bounds are closely related to the \textit{fill distance}, which is a space-filling property of the design. This suggests that it makes a good design by minimizing its fill distance. We also prove a similar error bound for universal kriging with a random kernel function. In addition, our theory shows that the maximum likelihood estimator for the regression coefficient of universal kriging can be \textit{inconsistent}, which is a new result to the best of our knowledge.

This paper is organized as follows. In Section \ref{sec:preliminaries}, we review the mathematical foundation of simple kriging and state the objectives of this paper. 
In Section \ref{sec:bounds}, we present our main results on the uniform error bounds for kriging predictors. Comparison with existing results in the literature is given in Section \ref{sec:comp}. Some simulation studies are presented in Section \ref{sec:simulation}, which confirm our theoretical analysis. We extend our theory from simple kriging to universal kriging in Section \ref{sec:new}. Concluding remarks and discussion are given in Section \ref{sec:discussion}. Appendix \ref{App:proof} contains the proof of Theorem \ref{Th:main}, the main theorem of this work.  Appendices \ref{app:B}-\ref{app:D} consist of the proofs of Theorems \ref{Th:BLUP}-\ref{Th:randomkernel}, respectively. Some necessary mathematical tools are reviewed in the supplementary materials.

\section{Preliminaries and motivation}\label{sec:preliminaries}

In Sections \ref{sec:review} and \ref{sec:interpolant}, we review the kriging method and introduce some proper notation. In Section \ref{sec:unanswered}, we state the primary goal of our work.

\subsection{Review on the simple kriging method}
\label{sec:review}

Let $Z(\bm x)$ be a Gaussian process on $\mathbf{R}^d$. In this work, we suppose that $Z$ has mean zero and is \textit{stationary}, i.e., the covariance function of $Z$ depends only on the difference between the two input variables. Specifically, we denote
\begin{eqnarray*}
\text{Cov}(Z(\bm x),Z(\bm x'))=\sigma^2 \Psi(\bm x-\bm x'),
\end{eqnarray*}
for any $\bm x,\bm x'\in\mathbf{R}^d$, where $\sigma^2$ is the variance and $\Psi$ is the correlation function.  The correlation function should be positive definite and satisfy $\Psi(0)=1$. In particular, we consider two important families of correlation functions. The isotropic Gaussian correlation function is defined as
\begin{eqnarray}\label{gaussian}
\Psi(\bm x;\phi)=\exp\{-\phi \|\bm x\|^2\},
\end{eqnarray}
with some $\phi>0$, where $\|\cdot\|$ denotes the Euclidean norm.
The isotropic Mat\'ern correlation function \citep{santner2013design,stein2012interpolation} is defined as
\begin{eqnarray}\label{matern}
\Psi(\bm x;\nu,\phi)=\frac{1}{\Gamma(\nu)2^{\nu-1}}(2\sqrt{\nu}\phi \|\bm x\|)^\nu K_\nu(2\sqrt{\nu}\phi\|\bm x\|),
\end{eqnarray}
where $\phi,\nu>0$ and $K_\nu$ is the modified Bessel function of the second kind. The parameter $\nu$ is often called the smoothness parameter, because it determines the smoothness of the Gaussian process \citep{cramer2013stationary}.

Suppose that we have observed $Z(\bm x_1),\ldots,Z(\bm x_n)$, in which $\bm x_1,\ldots,\bm x_n$ are distinct points. We shall use the terminology in design of experiments \citep{wu2011experiments} and call $\{\bm x_1,\ldots,\bm x_n\}$ the \textit{design points}, although in some situations (e.g., in spatial statistics and machine learning) these points are observed without the use of design. In this paper, we do not assume any (algebraic or geometric) structure for the design points $\{\bm x_1,\ldots,\bm x_n\}$. They are called \textit{scattered points} in the applied mathematics literature.

The aim of \textit{simple kriging} is to predict $Z(\bm x)$ at an untried $\bm x$ based on the observed data $Z(\bm x_1),\ldots,Z(\bm x_n)$, which is done by calculating the conditional distribution. It follows from standard arguments \citep{santner2013design,banerjee2004hierarchical} that, conditional on $Z(\bm x_1),\ldots,Z(\bm x_n)$, $Z(\bm x)$ is normally distributed, with
\begin{eqnarray}
&\mathbb{E}[Z(\bm x)|Z(\bm x_1),\ldots,Z(\bm x_n)]=\bm r^T(\bm x) \mathbf{K}^{-1} \bm Y,& a.s.,\label{mean}\\
&\text{Var}[Z(\bm x)|Z(\bm x_1),\ldots,Z(\bm x_n)]=\sigma^2(1-\bm r^T(\bm x) \mathbf{K}^{-1} \bm r(\bm x)), & a.s.,\label{variance}
\end{eqnarray}
where $\bm r(\bm x)=(\Psi(\bm x-\bm x_1),\ldots,\Psi(\bm x-\bm x_n))^T, \mathbf{K}=(\Psi(\bm x_j-\bm x_k))_{j k}$ and $\bm Y=(Z(\bm x_1),\ldots,Z(\bm x_n))^T$.

The conditional expectation $\mathbb{E}[Z(\bm x)|Z(\bm x_1),\ldots,Z(\bm x_n)]$ is a natural predictor of $Z(\bm x)$ using $Z(\bm x_1),\ldots,Z(\bm x_n)$, because it is the best linear predictor \citep{stein2012interpolation,santner2013design}. It is worth noting that a nice property of the Gaussian process models is that the predictor (\ref{mean}) has an explicit expression, which explains why kriging is so popular and useful.

The above simple kriging method can be extended. Instead of using a mean zero Gaussian process, one may introduce extra degrees of freedom by assuming that the Gaussian process has an unknown constant mean. More generally, one may assume the mean function is given by a linear combination of known functions. The corresponding methods are referred to as ordinary kriging and universal kriging, respectively. A standard prediction scheme is the best linear unbiased prediction \citep{santner2013design,stein2012interpolation}. In this work, we shall first consider simple kriging in Sections \ref{sec:interpolant}-\ref{sec:simulation}, and then extend our results to universal kriging in Section \ref{sec:new}. This organization is based on the following reasons: 1)
the predictive mean of simple kriging (\ref{mean}) is identical to the radial basis function interpolant (see Section \ref{sec:interpolant}), which is an important mathematical tool which our theory relies on, 2) our main theorem for simple kriging (Theorem \ref{Th:main}) requires less regularity conditions than those for universal kriging, 3) our theory for simple kriging, together with the techniques we develop to prove Theorem \ref{Th:main}, serves as a basis for establishing the results for universal kriging.

\subsection{Kriging interpolant}
\label{sec:interpolant}

The conditional expectation in (\ref{mean}) defines an interpolation scheme. To see this, let us suppress the randomness in the probability space and then $Z(\bm x)$ becomes a deterministic function, often called a sample path. It can be verified that, as a function of $\bm x$, $\bm r^T \mathbf{K}^{-1} \bm Y$ in (\ref{mean}) goes through each $Z(\bm x_j), j=1,\ldots,n$.

The above interpolation scheme can be applied to an arbitrary function $f$. Specifically, given design points $\mathbf{X}=(\bm x_1,\ldots,\bm x_n)$ and observations $f(\bm x_1),\ldots,f(\bm x_n)$, we define the \textit{kriging interpolant} by
\begin{eqnarray}\label{interpolant}
\mathcal{I}_{\Psi,\mathbf{X}}f(\bm x)=\bm r^T(\bm x) \mathbf{K}^{-1} \bm F,
\end{eqnarray}
where $\bm r(\bm x)=(\Psi(\bm x-\bm x_1),\ldots,\Psi(\bm x-\bm x_n))^T, \mathbf{K}=(\Psi(\bm x_j-\bm x_k))_{j k}$ and $\bm F=(f(\bm x_1),\ldots,f(\bm x_n))^T$.
This interpolation scheme is also refered to as the \textit{radial basis function} interpolation \citep{wendland2004scattered}.
The only difference between (\ref{interpolant}) and (\ref{mean}) is that we replace the Gaussian process $Z$ by a function $f$ here. In other words,
\begin{eqnarray}
\mathbb{E}[Z(\bm x)|Z(\bm x_1),\ldots,Z(\bm x_n)]=\mathcal{I}_{\Psi,\mathbf{X}}Z(\bm x), & a.s.
\end{eqnarray}

As mentioned in Section \ref{sec:review}, the conditional expectation $\mathbb{E}[Z(\bm x)|Z(\bm x_1),\ldots,Z(\bm x_n)]$ is a natural predictor of $Z(\bm x)$. A key objective of this work is to derive a uniform bound of the predictive error of the kriging method, given by
$Z(\bm x)-\mathbb{E}[Z(\bm x)|Z(\bm x_1),\ldots,Z(\bm x_n)]$,
which is equal to $Z(\bm x)-\mathcal{I}_{\Psi,\mathbf{X}}Z(\bm x)$ almost surely.

In practice, $\Psi$ is usually unknown. Thus it is desirable to develop a theory that also covers the cases with misspecified correlations.
In this work, we suppose that we use another correlation function $\Phi$ for prediction. We call $\Psi$ the \textit{true correlation function} and $\Phi$ the \textit{imposed correlation function}. Under the imposed correlation function, the kriging interplant of the underlying Gaussian process becomes $\mathcal{I}_{\Phi,\mathbf{X}}Z(\bm x)$. In this situation, the interpolant cannot be interpreted as the conditional expectation. With an abuse of terminology, we still call it a kriging predictor.


\subsection{Goal of this work}
\label{sec:unanswered}

Our aim is to study the approximation power of the kriging predictor. For simple kriging, we are interested in bounding the \textit{maximum prediction error} over a region $\Omega$,
\begin{eqnarray}\label{aim}
\sup_{\bm x\in\Omega} |Z(\bm x)-\mathcal{I}_{\Phi,\mathbf{X}}Z(\bm x)|
\end{eqnarray}
in a probabilistic manner,
where $\Omega$ is the region of interest, also called the experimental region, and $\Omega\supset\{\bm x_1\ldots,\bm x_n\}$. For universal kriging, our aim is to bound a quantity similar to (\ref{aim}), but in which $\mathcal{I}_{\Phi,\mathbf{X}}Z(\bm x)$ should be replaced by the best linear unbiased predictor, or a more general predictor given by universal kriging with an estimated kernel function.

Our obtained results on the error bound in \eqref{aim} can be used to address or answer the following three questions.

First, the quantity (\ref{aim}) captures the worst case prediction error of kriging.  In many practical problems, we are interested in recovering a whole function rather than predicting for just one point. Therefore, obtaining uniform error bounds are of interest because they provide some insight on how we can modify the pointwise error bound to achieve a uniform coverage.

Second, we study the case of misspecified correlation functions. This address a common question in kriging when the true correlation function is unknown: how to gain model robustness under a misspecified correlation function and how much efficiency loss is incurred.

Third, our framework allows the study of an arbitrary set of design points (also called scattered points). Thus our results can facilitate the study of kriging with fixed or random designs. In addition, our theory can be used to justify the use of space-filling designs \citep{santner2013design}, in which the design points spread (approximatedly) evenly in the design region.

\section{Uniform error bounds for Simple kriging}
\label{sec:bounds}

This section contains our main theoretical results on the prediction error of  simple kriging.

\subsection{Error bound in terms of the power function}\label{sec:power}

It will be shown that, the predictive variance (\ref{variance}) plays a curial role on the prediction error, when the true correlation function is known, i.e., $\Phi=\Psi$. To incorporate the case of misspecified correlation functions, we define
the \textit{power function} as
\begin{eqnarray}\label{power}
P^2_{\Phi,\mathbf{X}}(\bm x)=1-\bm r^T(\bm x) \mathbf{K}^{-1} \bm r(\bm x),
\end{eqnarray}
where $\bm r(\bm x)=(\Phi(\bm x-\bm x_1),\ldots,\Phi(\bm x-\bm x_n))^T$, and $ \mathbf{K}=(\Phi(\bm x_j-\bm x_k))_{j k}$.

The statistical interpretation of the power function is evident. From (\ref{variance}) it can be seen that, if $\Psi=\Phi$, the power function is the kriging predictive variance for a Gaussian process with $\sigma^2=1$. Clearly, we have $P^2_{\Phi,\mathbf{X}}(\bm x)\leq 1$.

To pursue a convergence result under the uniform metric, we define
\begin{eqnarray}\label{eq:PphiX}
P_{\Phi,\mathbf{X}}:=\sup_{\bm x\in\Omega}P_{\Phi,\mathbf{X}}(\bm x).
\end{eqnarray}

We now state the main results on the error bounds for kriging predictors. Recall that the prediction error under the uniform metric is given by (\ref{aim}).

The results depend on some smoothness conditions on the imposed kernel. Given any function $f$, let $\tilde{f}$ be its Fourier transform. According to the inversion formula in Fourier analysis, $\tilde{\Psi}/(2\pi)^d$ is the spectral density of the stationary process $Z$ if $\Psi$ is continuous and integrable on $\mathbf{R}^d$.

\begin{condition}\label{C1}
	The kernels $\Psi$ and $\Phi$ are continuous and integrable on $\mathbf{R}^d$, satisfying
  \begin{eqnarray}\label{smoothness}
  \|\tilde{\Psi}/\tilde{\Phi}\|_{L_\infty(\mathbf{R}^d)}=:A_1^2<+\infty.
  \end{eqnarray}
In addition, there exists $\alpha\in(0,1]$, such that
	\begin{eqnarray}\label{differentiability}
	\int_{\mathbf{R}^d} \|\bm \omega\|^\alpha\tilde{\Phi}(\bm \omega) d \bm  \omega=:A_0<+\infty.
	\end{eqnarray}
\end{condition}



%
%
Now we are able to state the first main theorem of this paper. Recall that $\sigma^2$ is the variance of $Z(\bm x)$. The proofs of Theorem \ref{Th:main} and the theorems in Section \ref{sec:new} make extensive use of the scattered data approximation theory and a maximum inequality for Gaussian processes. Detailed discussions of relevant areas are given in, for example, \cite{wendland2004scattered} and \cite{van1996weak}, respectively. We also collect the required mathematical tools and results in the supplementary material.

\begin{theorem}\label{Th:main}
	Suppose Condition \ref{C1} holds, and the design set $\mathbf{X}$ is dense enough in the sense that $P_{\Phi,\mathbf{X}}$ defined in (\ref{eq:PphiX}) is no more than some given constant $C$. Then for any $u>0$, with probability at least $1-2\exp\{-u^2/(2A_1^2\sigma^2 P_{\Phi,\mathbf{X}}^2)\}$, the kriging prediction error has the upper bound
		\begin{eqnarray}\label{eq:thmBound}
			\sup_{\bm  x\in\Omega}|Z(\bm x)-\mathcal{I}_{\Phi,\mathbf{X}}Z(\bm x)|\leq K \sigma P_{\Phi,\mathbf{X}}\log^{1/2}(e/P_{\Phi,\mathbf{X}})+u.
		\end{eqnarray}
Here the constants $C,K>0$ depend only on $\Omega,\alpha, A_0$ and $A_1$.
\end{theorem}

Theorem \ref{Th:main} presents an upper bound on the maximum prediction error of kriging. This answers the first question posed in Section \ref{sec:unanswered}. We will give more explicit error bounds in terms of the design $\mathbf{X}$ and the kernel $\Phi$ in Section \ref{sec:fill}. Theorem \ref{Th:main} can also be used to study the case of misspecified correlation functions, provided that condition (\ref{smoothness}) is fulfilled. Condition (\ref{smoothness}) essentially requires that the imposed correlation $\Phi$ is \textit{no smoother} than the true correlation function $\Psi$. Theorem \ref{Th:main} can also be used to address the third question posed in Section \ref{sec:unanswered}. Note that the right side of (\ref{eq:thmBound}) is a deterministic function depending on the design, and is decreasing in $P_{\Phi,\mathbf{X}}$ if $P_{\Phi,\mathbf{X}}$ is large enough. Therefore, it is reasonable to consider designs which minimize $P_{\Phi,\mathbf{X}}$. Such a construction depends on the specific form of $\Phi$. In Section \ref{sec:fill}, we will further show that, by maximizing certain space-filling measure, one can arrive at the optimal rate of convergence for a broad class of correlation functions.

From Theorem \ref{Th:main}, we also observe that the constant $A_1$ in (\ref{smoothness}) determines the decay rate of the maximum prediction error. In other words, the maximum prediction error appears more concentrated around its mean when the imposed kernel is closer to the true correlation function. Note that condition (\ref{differentiability}) requires a moment condition on the spectral density, which is fulfilled for any Mat\'ern or Gaussian kernel.

The non-asymptotic upper bound in Theorem \ref{Th:main} implies some asymptotic results which are of traditional interests in spatial statistics and related areas. For instance, suppose we adopt a classic setting of fixed-domain asymptotics  \citep{stein2012interpolation} in which the probabilistic structure of $Z(\bm x)$ and the kernel function $\Phi$ are fixed, and the number of design points increases so that $P_{\Phi,\mathbf{X}}$ tends to zero. 
Corollary \ref{Coro:1} is an immediate consequence of Theorem \ref{Th:main}, which shows the weak convergence and $L_p$ convergence of the maximum prediction error.

\begin{corollary}\label{Coro:1}
For fixed $\Psi,\Phi$, $\Omega$, and $\sigma$, we have the following asymptotic results
\begin{eqnarray}
\sup_{\bm x\in\Omega}|Z(\bm x)-\mathcal{I}_{\Phi,\mathbf{X}}Z(\bm x)|=O_\mathbb{P}(P_{\Phi,\mathbf{X}}\log^{1/2}(1/P_{\Phi,\mathbf{X}})),\label{rate}\\
\left(\mathbb{E}\left[\sup_{\bm x\in\Omega}|Z(\bm x)-\mathcal{I}_{\Phi,\mathbf{X}}Z(\bm x)|^p\right]\right)^{1/p}=O(P_{\Phi,\mathbf{X}}\log^{1/2}(1/P_{\Phi,\mathbf{X}})),\label{lp}
\end{eqnarray}
for any $1\leq p<+\infty$, as $P_{\Phi,\mathbf{X}}\rightarrow 0$.
\end{corollary}

\begin{proof}
Theorem \ref{Th:main} implies (\ref{rate}) directly. For (\ref{lp}), it follows from
\begin{align*}
& \mathbb{E}\left[\sup_{\bm x\in\Omega}|Z(\bm x)-\mathcal{I}_{\Phi,\mathbf{X}}Z(\bm x)|^p\right]\\
= & \int_0^\infty \mathbb{P}(\sup_{\bm x\in\Omega}|Z(\bm x)-\mathcal{I}_{\Phi,\mathbf{X}}Z(\bm x)|>t^{1/p}) d t\\
= & \bigg(\int_0^{[K\sigma P_{\Phi,\mathbf{X}}\log^{1/2}(e/P_{\Phi,X})]^p} + \int_{[K\sigma P_{\Phi,\mathbf{X}}\log^{1/2}(e/P_{\Phi,X})]^p}^\infty \bigg)\mathbb{P}(\sup_{\bm x\in\Omega}|Z(\bm x)-\mathcal{I}_{\Phi,\mathbf{X}}Z(\bm x)|>t^{1/p}) d t \\
\leq & \left[K\sigma P_{\Phi,\mathbf{X}}\log^{1/2}(e/P_{\Phi,X})\right]^p+\int_0^\infty 2\exp\{-t^{2/p}/(2A_1^2\sigma^2P_{\Phi,\mathbf{X}}^2)\} d t\\
 = & O(P_{\Phi,\mathbf{X}}^p\log^{p/2}(1/P_{\Phi,\mathbf{X}})),
\end{align*}
where the inequality follows from Theorem \ref{Th:main}.
\end{proof}

We believe that (\ref{rate}) and (\ref{lp}) are the full convergence rate because from (1.3) in the supplementary materials we can see that the convergence rate of the radial basis approximation for deterministic functions in the reproducing kernel Hilbert space is $O(P_{\Phi,\mathbf{X}})$ and these two rates are nearly at the same order of magnitude, expect for a logarithmic factor. This is reasonable because the support of a Gaussian process is typically larger than the corresponding reproducing kernel Hilbert space \citep{van2008reproducing}.
As said earlier in this section, if $\Psi=\Phi$, $P_{\Phi,\mathbf{X}}$ is the supremum of the pointwise predictive standard deviation. Thus Corollary \ref{Coro:1} implies that, if $\Psi$ is known, the predictive error of kriging under the uniform metric is not much larger than its pointwise error.


\subsection{Error bounds in terms of the fill distance}\label{sec:fill}

Our next step is to find error bounds which are easier to interpret and compute than that in Theorem \ref{Th:main}. To this end, we wish to find an upper bound of $P_{\Phi,\mathbf{X}}$, in which the effects of the design $\mathbf{X}$ and the kernel $\Phi$ can be made explicit and separately. This step is generally more complicated, but fortunately some upper bounds are available in the literature, especially for the Gaussian and the Mat\'ern kernels.
These bounds are given in terms of the \textit{fill distance}, which is a quantity depending only on the design $\mathbf{X}$. Given the experimental region $\Omega$, the fill distance of a design $\mathbf{X}$ is defined as
\begin{equation}\label{filldistance}
h_{\mathbf{X}}:=\sup_{\bm x\in\Omega}\min_{\bm x_j\in\mathbf{X}}\|\bm x-\bm x_j\|.
\end{equation}
Clearly, the fill distance quantifies the space-filling property \citep{santner2013design} of a design. A design having the minimum fill distance among all possible designs with the same number of points is known as a minimax distance design \citep{johnson1990minimax}.

The upper bounds of $P_{\Phi,\mathbf{X}}$ in terms of the fill distance for Gaussian and Mat\'ern kernels are given in Lemmas \ref{Th:Gaussian} and \ref{Th:Matern}, respectively.

\begin{lemma}[\citealp{wendland2004scattered}, Theorem 11.22]\label{Th:Gaussian}
	Let $\Omega = [0,1]^d$; $\Phi(x)$ be a Gaussian kernel given by (\ref{gaussian}). Then there exist constants $c,h_0$ depending only on $\Omega$ and the scale parameter $\phi$ in (\ref{gaussian}), such that
$P_{\Phi,\mathbf{X}}\leq h_{\mathbf{X}}^{c/h_{\mathbf{X}}}$ provided that $h_{\mathbf{X}}\leq h_0$.
\end{lemma}

\begin{lemma}[\citealp{wu1993local}, Theorem 5.14]\label{Th:Matern}
	Let $\Omega$ be compact and convex with a positive Lebesgue measure; $\Phi(x)$ be a Mat\'ern kernel given by (\ref{matern}) with the smoothness parameter $\nu$. Then there exist constants $c,h_0$ depending only on $\Omega,\nu$ and the scale parameter $\phi$ in (\ref{matern}), such that
	$P_{\Phi,\mathbf{X}}\leq c h_{\mathbf{X}}^{\nu}$ provided that $h_{\mathbf{X}}\leq h_0$.
\end{lemma}

Using the upper bounds of $P_{\Phi,\mathbf{X}}$ given in Lemmas \ref{Th:Gaussian} and \ref{Th:Matern}, we can further deduce error bounds of the kriging predictor in terms of the fill distance defined in (\ref{filldistance}). 
We demonstrate these results in Examples \ref{eg1}-\ref{eg3}.

\begin{eg}\label{eg1}
	Here we assume $\Phi$ is a Mat\'ern kernel in (\ref{matern}) with smoothness parameter $\nu$. It is known that
	\begin{eqnarray}\label{maternfourier}
	\tilde{\Phi}(\bm \omega)=2^d \pi^{d/2}\frac{\Gamma(\nu+d/2)}{\Gamma(\nu)}(4\nu \phi^2)^\nu (4\nu\phi^2+\|\bm{\omega}\|^2)^{-(\nu+d/2)},
	\end{eqnarray}
	where $\phi$ is the scale parameter in (\ref{matern}).
	See, for instance, \cite{wendland2004scattered,tuo2015efficient}. Suppose $\Psi$ is a Mat\'ern correlation function with smoothness $\nu_0$. It can be verified that Condition \ref{C1} holds if and only if $0<\nu\leq \nu_0$. Therefore, if $0<\nu\leq \nu_0$, we can invoke Lemma  \ref{Th:Matern} and Theorem \ref{Th:main} to obtain that the kriging predictor converges to the true Gaussian process with a rate at least $O_\mathbb{P}(h_{\mathbf{X}}^\nu\log^{1/2}(1/h_{\mathbf{X}}))$ as $h_{\mathbf{X}}$ tends to zero. It can be seen that the rate of convergence is maximized at $\nu=\nu_0$. In other words, if the true smoothness is known \textit{a priori}, one can obtain the greatest rate of convergence.
\end{eg}

\begin{eg}
	Suppose $\Phi$ is the same as in Example \ref{eg1}, and $\Psi$ is a Gaussian correlation function in (\ref{gaussian}), with spectral density \citep{santner2013design}
	$\tilde{\Psi}(\bm \omega)=(\pi/\phi)^{2/d}\exp\{-\|\bm \omega\|^2/(4\phi)\},$ where $\phi$ is the scale parameter in (\ref{gaussian}). Then Condition \ref{C1} holds for any choice of $\nu$. Then we can invoke Lemma  \ref{Th:Matern} and Theorem \ref{Th:main} to obtain the same rate of convergence as in Example \ref{eg1}.
\end{eg}


\begin{eg}\label{eg3}
	Suppose $\Phi=\Psi$, and $\Phi$ is a Gaussian kernel in (\ref{gaussian}). Then we can invoke Lemmas \ref{Th:Gaussian} and Theorem \ref{Th:main} to obtain the rate of convergence $O_\mathbb{P}(h_{\mathbf{X}}^{c/h_{\mathbf{X}}-1/2}\log^{1/2}(1/h_{\mathbf{X}}))$ for some constant $c>0$. Note that this rate is faster than the rates obtained in Examples \ref{eg1}-\ref{eg3}, because it decays faster than any polynomial of $h_{\mathbf{X}}$. Such a rate is known as a spectral convergence order \citep{xiu2010numerical,wendland2004scattered}.
\end{eg}

The upper bounds in Lemmas \ref{Th:Gaussian} and \ref{Th:Matern} explain more explicitly how the choice of designs can affect the prediction performance. Note that in Examples 1-3, the upper bounds are  increasing in $h_\mathbf{X}$. This suggests that we should consider the designs with a minimum $h_{\mathbf{X}}$ value, which are known as the maximin distance designs \citep{johnson1990minimax}. Therefore our theory shows that the maximin distance designs enjoy nice theoretical guarantees for all Gaussian and Mat\'ern kernels. In contrast with the designs minimizing $P_{\Phi,\mathbf{X}}$ as discussed after Theorem \ref{Th:main}, it would be practically beneficial to use the maximin distance designs because they can be constructed without knowing which specific kriging model is to be used.

\subsection{Comparison with some existing results}\label{sec:comp}

We make some remarks on the relationship between our results and some existing results. In Table \ref{Tab:comparison}, we list some  related results in the literature concerning the prediction of some underlying function, which is either a realization of a Gaussian process or a determinisitc function in a reproducing kernel Hilbert spaces.
It can be seen from Table \ref{Tab:comparison} that only \cite{seleznjev1999certain} and \cite{wu1993local} address the uniform convergence problem.

\begin{table}[!h]
	\centering
	\tiny
	\begin{tabular}{p{2cm} p{2.4cm} p{1.8cm} p{2.5cm} p{2.5cm} p{2.3cm}}
		\hline\hline
		Article/Book & {Model assumption} & Predictor &  {Design} & {Type of convergence} & {Rate of convergence (Mat\'ern kernels)} \\
		\hline\hline
		{present work} & {Gaussian process with misspecification} & kriging  & {scattered points}  & {$L_p$ conv., uniform in $x$}  & {$h_{\mathbf{X}}^\nu(\log (1/h_{\mathbf{X}}))^{1/2}$}\\
		\hline
		\cite{yakowitz1985comparison}  & {stochastic process with misspecification} & kriging & {scattered points} & {Mean square conv., pointwise in $x$}   & NA\\
		\hline
		{\cite{stein1990uniform}}  & {Stochastic process with misspecification} & kriging & {regular grid points} & {Mean square conv., pointwise in $x$}   & {$n^{-\nu/d}$}\\
		\hline
		{\cite{seleznjev1999certain}}  & {Gaussian process} & best linear approximation & {optimally chosen points in an interval} & {$L_p$ conv., uniform in $x$} & {$n^{-\nu}(\log n)^{1/2}$} \\
		\hline
		{\cite{ritter2007average}} & {Gaussian process} & kriging & {optimally chosen points} & {Mean square conv., $L_2$ in $x$} & {$n^{-\nu/d}$}
		\\
		\hline
		{\cite{wu1993local}} & {Deterministic function} & kriging & {scattered points} & {uniform in $x$} & {$h^{\nu}_{\mathbf{X}}$} \\
		\hline\hline
	\end{tabular}
	\caption{Comparison between our work and some existing results}
	\label{Tab:comparison}
\end{table}

\cite{seleznjev1999certain} study the rate of convergence of the best linear approximation under an optimally chosen points in an interval. In other words, the predictor is constructed using the best linear combination of the observed data. Thus this predictor is in general different from the kriging predictor. Also, note that their work is limited to the one dimensional case where the points are chosen in a specific way. Therefore, this theory does not directly address the question raised in this paper. However, their result, together with our findings in Example \ref{eg1}, does imply an interesting property of kriging. Recall that in Example \ref{eg1}, the rate of convergence for a (known) Mat\'ern correlation is at least $O_\mathbb{P}(h_{\mathbf{X}}^\nu\log^{1/2}(1/h_{\mathbf{X}}))$. If a space-filling design is used in an interval, then $h_{\mathbf{X}}\sim 1/n$ and the convergence rate is $O_\mathbb{P}(n^{-\nu}(\log n)^{1/2})$, which coincides with the best possible rate of convergence given by a linear predictor. Because the kriging predictor is a linear predictor, we can conclude that our uniform upper bound for kriging is sharp in the sense that it captures the acctual rate of convergence.

Among the papers listed in Table \ref{Tab:comparison}, \cite{wu1993local} is the only one comparable to ours, in the sense that they consider a uniform prediction error under a scatter set of design points.
They obtain error estimates of the kriging-type interpolants for a deterministic function, known as the radial basis function approximation. 
Although the mathematical formulations of the interpolants given by kriging and radial basis functions are similar, the two methods are different in their mathematical settings and assumptions. In radial basis function approximation, the underlying function is assumed \textit{fixed}, while kriging utilizes a probabilistic model, driven by a Gaussian random field. 

Kriging with misspecified correlation functions is discussed in \cite{yakowitz1985comparison,stein1988asymptotically,stein1990bounds,stein1990uniform}. It has been proven in these papers that some correlation functions, especially the Mat\'ern correlation family, are robust against model misspecification. However, they do not consider convergence under a uniform metric. More discussions on this point are given in Section \ref{sec:discussion}.

\section{Simulation studies}\label{sec:simulation}
In Example \ref{eg1}, we have shown that if $\Psi$ and $\Phi$ are Mat\'ern kernels with smoothness parameters $\nu_0$ and $\nu$, respectively, and $\nu\leq \nu_0$, then the kriging predictor converges with a rate at least $O_\mathbb{P}(h_{\mathbf{X}}^\nu\log^{1/2}(1/h_{\mathbf{X}}))$. In this section we report simulation studies that verify that this rate is sharp, i.e., the true convergence rate coincides with that given by the theoretical upper bound.





We denote the expectation of the left-hand side of (\ref{eq:thmBound}) by $\mathcal{E}$. If the error bound (\ref{eq:thmBound}) is sharp, we have the approximation $$\mathcal{E} \approx c h_{\mathbf{X}}^\nu\log^{1/2}(1/h_{\mathbf{X}})$$ for some constant $c$ independent of $h_{\mathbf{X}}$. Taking logarithm on both sides of the above formula yields
\begin{eqnarray}\label{eq:log}
\log \mathcal{E}\approx \nu\log h_{\mathbf{X}} + \frac{1}{2}\log(-\nu\log h_{\mathbf{X}}) + \log c.
\end{eqnarray}
Since $\log(-\nu\log h_{\mathbf{X}})$ is much smaller than $\log h_{\mathbf{X}}$, the effect of $\log(-\nu\log h_{\mathbf{X}})$ is negligible in (\ref{eq:log}). Consequently, we get our second approximation
\begin{eqnarray}\label{eq:linearreg}
\log \mathcal{E}\approx \nu\log h_{\mathbf{X}} + \log c.
\end{eqnarray}
As shown in (\ref{eq:linearreg}), $\log \mathcal{E}$ is approximately a linear function in $\log h_{\mathbf{X}}$ with slope $\nu$. Therefore, to assess whether (\ref{eq:thmBound}) is sharp, we should verify if the regression coefficient (slope) of $\log \mathcal{E}$ with respect to $\log h_{\mathbf{X}}$ is close to $\nu$.

In our simulation studies, the experimental region is chosen to be $\Omega = [0,1]^2$. To estimate the regression coefficient $\nu$ in (\ref{eq:linearreg}), we choose 50 different maximin Latin hypercube designs \citep{santner2013design} with sample sizes $10k$, for $k=1,2...,50$.
Note that each design corresponds to a specific value of the fill distance $h_{\mathbf{X}}$. For each $k$, we simulate the Gaussian processes 100 times to reduce the simulation error. For each simulated Gaussian process, we compute $\sup_{\bm x \in \Omega_1}|Z(\bm x)-\mathcal{I}_{\Phi,\mathbf{X}}Z(\bm x)|$ to approximate the sup-error $\sup_{\bm x \in \Omega}|Z(\bm x)-\mathcal{I}_{\Phi,\mathbf{X}}Z(\bm x)|$, where $\Omega_1$ is the set of grid points with grid length $0.01$. This should give a good approximation since the grid is dense enough. Next, we calculate the average of $\sup_{\bm x \in \Omega_1}|Z(\bm x)-\mathcal{I}_{\Phi,\mathbf{X}}Z(\bm x)|$ over the 100 simulations to approximate $\mathcal{E}$. Then the regression coefficient is estimated using the least squares method.

We conduct four simulation studies with different choices of the true and imposed smoothness of the Mat\'ern kernels, denoted by $\nu_0$ and $\nu$, respectively. We summarize the simulation results in Table \ref{Tab:simuResults}.





\begin{table}[h]
\centering
\begin{tabular}{c|c|c|c|c}
\hline
$\nu_0$ & $\nu$ &  Regression coefficient & Theoretical assertion & Relative difference\\
\hline
3 & 2.5 & 2.697 & 2.5 & 0.0788\\
5 & 3.5 & 3.544 & 3.5 & 0.0126 \\
3.5 & 3.5 & 3.582 & 3.5 & 0.0234\\
5 & 5 & 4.846 & 5 & 0.0308\\
\hline
\end{tabular}
\caption{Numerical studies on the convergence rates of kriging prediction. (The first two columns show the true and imposed smoothness parameters of the Mat\'ern kernels. The third column shows the convergence rate obtained from the simulation. The fourth column shows the convergence rate given by Theorem \ref{Th:main}. The last column shows the relative difference between the third and the fourth columns, given by |regression coefficient-theoretical assertion|/(theoretical assertion).)}
\label{Tab:simuResults}
\end{table}

It is seen in Table \ref{Tab:simuResults} that the regression coefficients are close to the values given by our theoretical analysis, with relative error no more than 0.08. By comparing the third and the fourth rows of Table \ref{Tab:simuResults}, we find that the regression coefficient does not have a significant change when $\nu$ remains the same, even if $\nu_0$ changes. On the other hand, the third and the fifth rows show that, the regression coefficient changes significantly as $\nu$ changes, even if $\nu_0$ keeps unchanged. This shows convincingly that the convergence rate is independent of the true smoothness of the Gaussian process, and the rate given by Theorem \ref{Th:main} is sharp. Note that our simulation studies justify the use of the leading term $\log  h_{\mathbf{X}}$ in (\ref{eq:log}) to assess the convergence rate but they do not cover the second term $\log (-\nu\log h_{\mathbf{X}})$, which is of lower order. Figure 1 in the supplementary material shows that the regression line for the logarithm of the fill distance and the logarithm of the average prediction error fits the data very well.

From the simulation studies, we can see that if the smoothness of the imposed kernel is lower, the kriging predictor converges slower. Therefore, to maximize the prediction efficiency, it is beneficial to set the smoothness parameter of the imposed kernel the same as the true correlation function.

	\section{Extensions to universal kriging}\label{sec:new}
	
	In this section, we extend the main result in Theorem \ref{Th:main} from simple kriging to universal kriging.
	As an extension of simple kriging, universal kriging is widely used in practice. Instead of using a zero mean Gaussian process, universal kriging assumes that the Gaussian process has a non-zero mean function, modeled as a linear combination of a set of basis functions with unknown regression coefficients. Specifically, we consider the following model:
	\begin{eqnarray}
	Y(\bm x)=\bm f^T(\bm x)\bm \beta+Z(\bm x),
	\end{eqnarray}
	where $f(\bm x):=(f_1(\bm x),\ldots,f_p(\bm x))^T$ is a vector of $p$ linearly independent known functions over $\Omega$, $\bm \beta$ is an unknown vector of regression coefficients, and $Z(\bm x)$ is a zero mean stationary Gaussian process with correlation function $\Psi$. The goal of universal kriging is to reconstruct $Y(\bm x)$ based on scattered observations $Y(\bm x_1),\ldots,Y(\bm x_n)$.
	
	A common practice is to use the maximum likelihood estimation (MLE) method to estimate $\bm \beta$ and the best linear unbiased predictor (BLUP) to predict $Y(\bm x)$ at an untried $\bm x$. The following facts can be found in the book by \cite{santner2013design}. As before, let $\mathbf{K}:=(\Psi(\bm x_j-\bm x_k))_{jk}$. Define $\mathbf{F}:=(\bm f(\bm x_1),\ldots,\bm f(\bm x_n))^T$. We temporarily suppose that $\Psi$ is known. Then the MLE of $\bm \beta$ is given by
	\begin{eqnarray}\label{betaMLE}
	\hat{\bm \beta}=(\mathbf{F}^T \mathbf{K}^{-1} \mathbf{F})^{-1} \mathbf{F}^T \mathbf{K}^{-1}\mathbf{Y},
	\end{eqnarray}
	with $\mathbf{Y}:=(Y(\bm x_1),\ldots,Y(\bm x_n))^T$.
	To use (\ref{betaMLE}), we should require $n\geq p$ so that $\mathbf{F}^T \mathbf{K}^{-1} \mathbf{F}$ is invertible. The BLUP of $Y(\bm x)$ is
	\begin{align}\label{BLUP}
		\hat{Y}_{\mathbf{BLUP}}(\bm x)= & [\bm f^T(\bm x)(\mathbf{F}^T \mathbf{K}^{-1} \mathbf{F})^{-1} \mathbf{F}^T \mathbf{K}^{-1}\nonumber\\
	& +\bm r^T(\bm x)\mathbf{K}^{-1}(\mathbf{I}_n-\mathbf{F}(\mathbf{F}^T\mathbf{K}^{-1}\mathbf{F})^{-1}\mathbf{F}^T\mathbf{K}^{-1})]\mathbf{Y},
	\end{align}
	where $\bm r(\bm x):=(\Psi(\bm x-\bm x_1),\ldots,\Psi(\bm x-\bm x_n))^T$.
	
	As before, we are interested in the situation with a misspecified correlation function, also denoted by $\Phi$. Using (\ref{BLUP}), we can calculate the ``BLUP'' of $Y(\bm x)$ under $\Phi$, denoted by $\hat{Y}_{\mathbf{BLUP},\Phi}(\bm x)$, with redefined $\bm r$ and $\mathbf{K}$ given by $\bm r:=(\Phi(\bm x-\bm x_1),\ldots,\Phi(\bm x-\bm x_n))^T$ and $\mathbf{K}:=(\Phi(\bm x_j-\bm x_k))_{jk}$.
	
	The goal of our theoretical study is to bound $\sup_{\bm x\in \Omega}|Y(\bm x)-\hat{Y}_{\mathbf{BLUP},\Phi}(\bm x)|$ in a way similar to Theorem \ref{Th:main}. The results are given in Theorem \ref{Th:BLUP}. We denote the minimum eigenvalue of a matrix $\mathbf{H}$ by $\lambda_{min}(\mathbf{H})$. Let $\mathcal{A}$ be the set of $p\times p$ submatrices of $\mathbf{F}$.
	
	\begin{theorem}\label{Th:BLUP}
		Suppose the conditions of Theorem \ref{Th:main} are fulfilled. In addition, $f_j\in\mathcal{N}_\Phi(\Omega)$ for $j=1,\ldots,p$. Then
		\begin{eqnarray}\label{theorem2}
		\mathbb{E}\left[\sup_{\bm x\in \Omega}|Y(\bm x)-\hat{Y}_{\mathbf{BLUP},\Phi}(\bm x)|\right]=O(P_{\Phi,\mathbf{X}}[pA+\log^{1/2}(1/P_{\Phi,\mathbf{X}})]),
		\end{eqnarray}
		where the asymptotic constant is independent of $\mathbf{X}$, $p$ and $f_j$'s; and
		$$A=\left(\sum_{j=1}^p\|f_j\|^2_{\mathcal{N}_\Phi(\Omega)}/\max_{\mathbf{F}_p\in\mathcal{A}}\lambda_{min}(\mathbf{F}_p^T \mathbf{F}_p)\right)^{1/2}.$$
	\end{theorem}
	
	The condition $f_i\in\mathcal{N}_\Phi(\Omega)$ in Theorem \ref{Th:BLUP} is mild if $\Phi$ is a Mat\'ern kernel. It is known that in such case, $\mathcal{N}_\Phi(\Omega)$ coincides with a Sobolev space, which contains all smooth functions, such as polynomials. See Corollary 10.13 of \cite{wendland2004scattered}.

Compared to the uniform error bound for simple kriging, (\ref{theorem2}) has an additional term $O(P_{\Phi,\mathbf{X}}pA)$, which is caused by the unknown regression coefficient $\beta$.
In many situations, $pA$ is bounded above by a constant, e.g., when $p$ and $f_j$'s are fixed and $\bm x_j$'s are independent random samples.
In this case
$O(P_{\Phi,\mathbf{X}}pA)=O(P_{\Phi,\mathbf{X}})$ and can be absorbed by the simple kriging uniform error bound.

As a by-product of our analysis, we show that the MLE $\hat{\bm \beta}$ is \textit{inconsistent} when $\Psi$ is known. The result in Theorem \ref{Th:inconsistency} shows that the covariance matrix of $\hat{\bm \beta}-\bm \beta$ is no less than the inverse of the inner product matrix of $f_j$'s in the reproducing kernel Hilbert space, in the sense that the former subtracting the latter is positive semidefinite.

\begin{theorem}\label{Th:inconsistency}
If $\Phi=\Psi$ and $f_j\in\mathcal{N}_\Phi(\Omega)$ for $j=1,\ldots,p$, we have
$\text{Var}(\hat{\bm \beta}-\bm \beta)\geq \mathbf{V}^{-1}$, where $\mathbf{V}:=(\langle f_j,f_k\rangle_{\mathcal{N}_\Phi(\Omega)})_{j k}$ is positive definite.
\end{theorem}

The proof of Theorem \ref{Th:inconsistency} shows that the estimation of $\bm \beta$ is perturbed by the Gaussian process, and thus becomes inconsistent. In addition, the well-known theory by \cite{ying1991asymptotic} and \cite{zhang2004inconsistent} suggests that the model parameters in the covariance functions may not have consistent estimators. Therefore, it would be more meaningful to use Gaussian process models for prediction, rather than for parameter identification.

Next, we study the uniform error bound when a \textit{random} kernel function is used. Such a result can be useful when an estimated correlation function is used. The main idea here is to study a more sophisticated type of uniform error, which also takes supremum over a family of correlation functions.

Let $\Phi_{\bm \theta}$ be a family of correlation functions indexed by $\bm \theta=(\theta_1,\ldots,\theta_q)^T\in{\bm \Theta}$. Suppose $\Theta$ is a compact subregion of $\mathbf{R}^q$. For notational simplicity, denote $P_{\mathbf{X}}=\max_{\bm \theta\in\Theta}P_{\Phi_{\bm \theta},\mathbf{X}}$. Theorem \ref{Th:randomkernel} provides a uniform error bound in terms of both untried $\bm x$ and parameter $\bm \theta$.

\begin{condition}\label{C1a}
	The kernels $\Phi_{\bm \theta}$ are continuous and integrable on $\mathbf{R}^d$, satisfying
  \begin{eqnarray}\label{smoothnesstheta}
  \|\tilde{\Psi}/\tilde{\Phi}_{\bm \theta}\|_{L_\infty(\mathbf{R}^d)}=:A_1^2<+\infty,
  \end{eqnarray}
  and
	\begin{eqnarray}\label{differentiabilitytheta}
	\int_{\mathbf{R}^d} \|\bm \omega\|^\alpha\tilde{\Phi}_{\bm \theta}(\bm \omega) d \bm  \omega=:A_0<+\infty,
	\end{eqnarray}
	for constants $\alpha\in(0,1], A_0,A_1$ independent of $\bm \theta$.
\end{condition}

\begin{theorem}\label{Th:randomkernel}
Suppose the conditions of Theorem \ref{Th:BLUP} and Condition \ref{C1a} are fulfilled.
Suppose $\Phi_{\bm \theta}(\bm x)$ is differentiable in $\bm \theta$ for each $\bm x$. In addition, suppose the differentiation in $\bm \theta$ and the Fourier transform in $\bm x$ are interchangeable, i.e.,
\begin{eqnarray}\label{interchange}
\widetilde{ \frac{\partial\Phi_{\bm \theta}}{\partial \theta_j}}=\frac{\partial \widetilde{\Phi_{\bm\theta}}}{\partial \theta_j}, j=1,\ldots,q.
\end{eqnarray}
Moreover, suppose
\begin{eqnarray}\label{logbound}
\sup_{\bm x\in\mathbf{R}^d,\bm \theta\in\Theta} \left|\frac{\partial}{\partial \theta_j}\log \widetilde{\Phi_{\bm \theta}}\right|\leq A_2<+\infty, j=1,\ldots,q.
\end{eqnarray}
Then,
\begin{eqnarray}\label{randomkernel}
\mathbb{E}\left[\sup_{\bm x\in \Omega,\bm \theta\in\Theta}|Y(\bm x)-\hat{Y}_{\mathbf{BLUP},\Phi_{\bm \theta}}(\bm x)|\right]=O(P_{\mathbf{X}}[pA+\log^{1/2}(1/P_{\mathbf{X}})]),
\end{eqnarray}
where $A$ is the same as in Theorem \ref{Th:BLUP} and the asymptotic constant is independent of $\mathbf{X}$, $p$ and $f_j$'s.
\end{theorem}

The uniform error bound in (\ref{randomkernel}) can govern the error bound when a random kernel is used. Specifically, suppose a random kernel, denoted by $\Phi_{\hat{\bm \theta}}$, is used, and the support of the random variable (estimator) $\hat{\bm \theta}$ is $\Theta$. Then we have $\mathbb{E}\sup_{\bm x\in \Omega}|Y(\bm x)-\hat{Y}_{\mathbf{BLUP},\Phi_{\hat{\bm \theta}}}(\bm x)|=O(P_{\mathbf{X}}[pA+\log^{1/2}(1/P_{\mathbf{X}})])$.

We now verify the conditions of Theorem \ref{Th:randomkernel} for Mat\'ern and Gaussian kernels. Suppose $\Psi$ is a Mat\'ern kernel with smoothness $\nu_0$. Let $\bm \theta=(\phi,\nu)$ and $\Theta$ is a compact subregion of $(0,+\infty)\times (0,\nu_0)$. Clearly, Condition \ref{C1a} is fulfilled. Recall that $\tilde{\Phi}_{\bm \theta}$ is given in (\ref{maternfourier}). Dominated convergence theorem ensures (\ref{interchange}) and we can verify (\ref{logbound}) via direct calculations. Suppose $\Psi$ is a Gaussian kernel in (\ref{gaussian}) with $\phi=\phi_0$. Suppose $\bm \theta=\phi$ and $\Theta$ is a compact subregion of $[\phi_0,+\infty)$. Similar direct calculations can show that the conditions in Theorem \ref{Th:randomkernel} are also fulfilled in this case.

The proofs of Theorems \ref{Th:BLUP} and \ref{Th:randomkernel} utilize the techniques we developed for proving Theorem \ref{Th:main}. These theorems show that the general rate of convergence $O_P(P_{\mathbf{X},\Phi}\log^{1/2}(1/P_{\mathbf{X},\Phi}))$ is still valid even if estimated mean and covariance functions are used.






\section{Conclusions and further discussion}\label{sec:discussion}

We first summarize the statistical implications of this work. We prove that the kriging predictive error converges to zero under a uniform metric, which justifies the use of kriging as a function reconstruction tool. Our analysis covers both simple and universal kriging. Kriging with a misspecified correlation function is also studied. Theorem \ref{Th:main} shows that there is a tradeoff between the predictive efficiency and the robustness. Roughly speaking, a less smooth correlation function is more robust against model misspecification. However, the price for robustness is to incur a small loss in prediction efficiency. With the help of the classic results in radial basis function approximation (in Lemmas \ref{Th:Gaussian} and \ref{Th:Matern}), we find that the predictive error of kriging is associated with the fill distance, which is a space-filling measure of the design. This justifies the use of space-filling designs for (stationary) kriging models.




We have proved in Theorem \ref{Th:main} that the kriging predictor is consistent if the imposed correlation function is \textit{undersmoothed}, i.e., the imposed correlation function is no smoother than the true correlation function. One would ask whether a similar result can be proven for the case of oversmoothed correlation functions. \cite{yakowitz1985comparison} proved that kriging with an oversmoothed Mat\'ern correlation function also achieves (pointwisely) predictive consistency. In light of this result, we may consider extensions of Theorem \ref{Th:main} to the oversmoothed case in a future work.

In a series of papers, \cite{stein1988asymptotically,stein1990bounds,stein1990uniform,stein1993simple} investigated the asymptotic efficiency of the kriging predictor. The theory in our work does not give assertions about prediction efficiency, although we provide explicit error bounds for kriging predictors with scattered design points in general dimensions. Another possible extension of this work is to consider the impact of a misspecified mean function. 
\cite{jiang2011best} address this problem in the context of small area estimation. 


\begin{center}
\Large \textbf{Appendix}
\end{center}

\begin{appendix}

\section{Proof of Theorem \ref{Th:main}}\label{App:proof}

Because $\mathcal{I}_{\Phi,\mathbf{X}}$ is a linear map between two functions, $\mathcal{I}_{\Phi,\mathbf{X}}Z(\bm x)$ is also a Gaussian process. Therefore, the problem in (\ref{aim}) is to bound the maximum value of a Gaussian process. The main idea of the proof is to invoke a maximum inequality for Gaussian processes, which states that the supremum of a Gaussian process is no more than a multiple of the integral of the covering number with respect to the natural distance $\mathfrak{d}$. The details are given in the supplementary materials. Also see \cite{adler2009random,van1996weak} for related discussions.

Without loss of generality, assume $\sigma = 1$, because otherwise we can consider the upper bound of $\sup_{\bm x\in\Omega}|Z(\bm x)-\mathcal{I}_{\Phi,\mathbf{X}}Z(\bm x)|/\sigma$ instead. Let $g(\bm x) = Z(\bm x)-\mathcal{I}_{\Phi,\mathbf{X}}Z(\bm x)$.
For any $\bm x,\bm x'\in \Omega$,
\begin{align*}
\mathfrak{d}(\bm x, \bm x')^2 = &\mathbb{E}(g(\bm x)-g(\bm x'))^2\\
                      = & \mathbb{E}(Z(\bm x)-\mathcal{I}_{\Phi,\mathbf{X}}Z(\bm x)-(Z(\bm x')-\mathcal{I}_{\Phi,\mathbf{X}}Z(\bm x')))^2 \\
                      = &\Psi(\bm x-\bm x) - 2\bm r^T(\bm x)\mathbf{K}^{-1}\bm r_1(\bm x) + \bm r^T(\bm x)\mathbf{K}^{-1}\mathbf{K}_1 \mathbf{K}^{-1}\bm r(\bm x)\\
          & + \Psi(\bm x'-\bm x') - 2\bm r^T(\bm x')\mathbf{K}^{-1}\bm r_1(\bm x') + \bm r^T(\bm x')\mathbf{K}^{-1}\mathbf{K}_1 \mathbf{K}^{-1}\bm r(\bm x')\\
          & -2[\Psi(\bm x-\bm x')- \bm r^T(\bm x')\mathbf{K}^{-1}\bm r_1(\bm x) - \bm r_1^T(\bm x')\mathbf{K}^{-1}\bm r(\bm x) + \bm r^T(\bm x)\mathbf{K}^{-1}\mathbf{K}_1 \mathbf{K}^{-1}\bm r(\bm x') ],
\end{align*}
where $\bm r_1(\cdot) = (\Psi(\cdot-\bm x_1),...,\Psi(\cdot-\bm x_n))^T$, $\bm r(\cdot) = (\Phi(\cdot-\bm x_1),...,\Phi(\cdot-\bm x_n))^T$, $\mathbf{K}_1 = (\Psi(\bm x_j-\bm x_k))_{jk}$, and $\mathbf{K} = (\Phi(\bm x_j-\bm x_k))_{jk}$.

The rest of our proof consists of the following steps. In step 1, we bound the covering number $N(\epsilon,\Omega,\mathfrak{d})$.
Next we bound the diameter $D$. In step 3, we invoke Lemma 1 in the supplementary materials to obtain a bound for the entropy integral. In the last step, we use (1.8) in the supplementary materials to obtain the desired results.

\noindent\textbf{Step 1: Bounding the covering number}

Let $h(\cdot) = \Psi(\bm x-\cdot) - \Psi(\bm x'-\cdot)$ and $h_1(\cdot) = \bm r^T(\bm x)\mathbf{K}^{-1}\bm r_1(\cdot) -  \bm r^T(\bm x')\mathbf{K}^{-1}\bm r_1(\cdot)$. It can verified that
\begin{align*}
\begin{split}
\mathfrak{d}(\bm x, \bm x')^2  = & - [h( \bm x') - \mathcal{I}_{\Phi,\mathbf{X}}h(\bm x')]+[h(\bm x) - \mathcal{I}_{\Phi,\mathbf{X}}h( \bm x)] \\
& + [h_1( \bm x') - \mathcal{I}_{\Phi,\mathbf{X}}h_1( \bm x')]-[h_1( \bm x) - \mathcal{I}_{\Phi,\mathbf{X}}h_1( \bm x)].
\end{split}
\end{align*}
By Condition \ref{C1},
$h\in \mathcal{N}_\Phi(\mathbf{R}^d)$, since $\Psi(\bm x-\cdot)\in \mathcal{N}_\Phi(\mathbf{R}^d)$ for any $\bm x\in \Omega$. Thus, by (1.3) in the supplementary materials,
\begin{align}\label{eq:dist2}
\mathfrak{d}(\bm x, \bm x')^2 \leq 2 P_{\Phi,\mathbf{X}}(\|h\|_{\mathcal{N}_\Phi(\mathbf{R}^d)}+\|h_1\|_{\mathcal{N}_\Phi(\mathbf{R}^d)}).
\end{align}


By Theorem 1 in the supplementary materials,
\begin{align}\label{eq:Fourierh}
\|h\|^2_{\mathcal{N}_\Phi(\mathbf{R}^d)} = (2\pi)^{-d}\int_{\mathbf{R}^d} \frac{|\tilde h(\bm\omega)|^2}{\tilde \Phi(\bm\omega)}d\bm\omega.
\end{align}
Using Condition \ref{C1} and (\ref{eq:Fourierh}), we obtain
\begin{align}\label{eq:FourierhC1}
\|h\|^2_{\mathcal{N}_\Phi(\mathbf{R}^d)} = (2\pi)^{-d}\int_{\mathbf{R}^d} \frac{|\tilde h(\bm\omega)|^2}{\tilde \Phi(\bm\omega)}d\bm\omega\leq A_1^2(2\pi)^{-d}\int_{\mathbf{R}^d} \frac{|\tilde h(\bm\omega)|^2}{\tilde \Psi(\bm\omega)}d\bm\omega = A_1^2 \|h\|^2_{\mathcal{N}_\Psi(\mathbf{R}^d)}.
\end{align}

We need the following inequality to bound $\|h\|^2_{\mathcal{N}_\Psi(\mathbf{R}^d)}$. For any $0<\beta\leq 1$ and $x\in\mathbf{R}$, we have
\begin{eqnarray}\label{cosineq}
|1-\cos x|\leq 2|x|^\beta.
\end{eqnarray}
This inequality is trivial when $|x|\geq 1$ because $|1-\cos x|\leq 2$; and for the case that $|x|<1$, (\ref{cosineq}) can be proven using the mean value theorem and the fact that $|x|\leq |x|^\beta$.
Note that the definition of $h$ implies that $\|h\|^2_{\mathcal{N}_\Psi(\mathbf{R}^d)} = \Psi(\bm x-\bm x) - 2\Psi(\bm x'-\bm x) + \Psi(\bm x'-\bm x')$.
Thus, by the Fourier inversion theorem and (\ref{cosineq}), we have
\begin{align}
\|h\|^2_{\mathcal{N}_\Psi(\mathbf{R}^d)} & = \Psi(\bm x-\bm x) - 2\Psi(\bm x'-\bm x) + \Psi(\bm x'-\bm x')\nonumber\\
& = 2(2\pi)^{-d}\int_{\mathbf{R}^d}(1-e^{i(\bm x-\bm x')^T\bm \omega})\tilde \Psi(\bm\omega)d\bm\omega\nonumber\\
& \leq \bigg( 4(2\pi)^{-d}\int_{\mathbf{R}^d}\|\omega\|^{\beta}\tilde \Psi(\bm\omega)d\bm\omega\bigg)\|\bm x-\bm x'\|^\beta\label{eq:hnorm1}\\
&=:C_1\|\bm x-\bm x'\|^{\beta},\label{eq:entropyC1}
\end{align}
for any $0<\beta\leq \alpha$. In particular, we now choose $\beta=\alpha/2$.
Now we consider $h_1(\cdot)$. It follows from a similar argument that $\|h_1\|^2_{\mathcal{N}_\Phi(\mathbf{R}^d)} \leq A_1^2 \|h_1\|^2_{\mathcal{N}_\Psi(\mathbf{R}^d)}$. The definition of $h_1$ implies $\|h_1\|^2_{\mathcal{N}_\Psi(\mathbf{R}^d)} = (\bm r(\bm x')-\bm r(\bm x))^T\mathbf{K}^{-1}\mathbf{K}_1 \mathbf{K}^{-1}(\bm r(\bm x')-\bm r(\bm x))$.

For any $\bm u= (u_1,...,u_n)^T$, the Fourier inversion theorem and Condition \ref{C1} yield
\begin{align}\label{eq:identity}
& \sum_{j,k=1}^n u_j\bar u_k \Psi(\bm x_j-\bm x_k)\nonumber\\
= & \frac{1}{(2\pi)^d}\int_{\mathbb{R}^d}\sum_{j,k=1}^n u_j\bar u_k e^{i(\bm  x_j-\bm x_k)^T\bm \omega }\tilde\Psi(\bm\omega)d\bm\omega\nonumber\\
= & \frac{1}{(2\pi)^d}\int_{\mathbb{R}^d}\bigg|\sum_{j=1}^n u_j e^{i\bm x_j^T\bm \omega}\bigg|^2\tilde\Psi(\bm\omega)d\bm\omega\\
\leq & \frac{A_1^2}{(2\pi)^d}\int_{\mathbb{R}^d}\bigg|\sum_{j=1}^n u_j e^{i\bm x_j^T\bm \omega}\bigg|^2\tilde\Phi(\bm\omega)d\bm\omega\nonumber\\
= & A_1^2\sum_{j,k=1}^n u_j\bar u_k \Phi(\bm x_j-\bm x_k).\nonumber
\end{align}
Then we choose $\bm u = \mathbf{K}^{-1}(\bm r(\bm x')-\bm r(\bm x))$ to get
\begin{align}\label{eq:h1norm2}
\|h_1\|^2_{\mathcal{N}_\Psi(\mathbf{R}^d)} \leq A_1^2(\bm r(\bm x')-\bm r(\bm x))^T \mathbf{K}^{-1}(\bm r(\bm x')-\bm r(\bm x)).
\end{align}

Let $h_2(\cdot) = \Phi(\cdot - \bm x')-\Phi(\cdot - \bm x)$. Then $\mathcal{I}_{\Phi,\mathbf{X}}h_2(\cdot) = \bm r^T(\cdot) \mathbf{K}^{-1}(\bm r(\bm x')-\bm r(\bm x))$. By (1.3) in the supplementary materials and the fact that $\|h_2\|^2_{\mathcal{N}_\Phi(\mathbf{R}^d)} = \Phi(\bm x-\bm x) - 2\Phi(\bm x'-\bm x) + \Phi(\bm x'-\bm x')$, we have
\begin{align}\label{eq:h1norm3}
& (\bm r(\bm x')-\bm r(\bm x))^T \mathbf{K}^{-1}(\bm r(\bm x')-\bm r(\bm x))\nonumber\\
\leq &  |h_2(\bm x') - \mathcal{I}_{\Phi,\mathbf{X}}h_2(\bm x')| + |h_2(\bm x) - \mathcal{I}_{\Phi,\mathbf{X}}h_2(\bm x)| + |h_2(\bm x')| + |h_2(\bm x)|\nonumber\\
\leq & 2P_{\Phi,\mathbf{X}}\sqrt{\Phi(\bm x-\bm x) - 2\Phi(\bm x'-\bm x) + \Phi(\bm x'-\bm x')} + 2(\Phi(\bm x-\bm x) - 2\Phi(\bm x'-\bm x) + \Phi(\bm x'-\bm x'))\nonumber\\
\leq & 2(P_{\Phi,\mathbf{X}} + \sqrt{\Phi(\bm x-\bm x) - 2\Phi(\bm x'-\bm x) + \Phi(\bm x'-\bm x')}) \sqrt{\Phi(\bm x-\bm x) - 2\Phi(\bm x'-\bm x) + \Phi(\bm x'-\bm x')}\nonumber\\
\leq & 2(P_{\Phi,\mathbf{X}} + 2) \sqrt{\Phi(\bm x-\bm x) - 2\Phi(\bm x'-\bm x) + \Phi(\bm x'-\bm x')}.
\end{align}
Thus, if $P_{\Phi,\mathbf{X}}<1$, by a similar argument in (\ref{eq:hnorm1}) (with the choice $\beta=\alpha$), and together with (\ref{eq:h1norm2}) and (\ref{eq:h1norm3}), we have
\begin{align}\label{eq:entropyC1X}
\|h_1\|^2_{\mathcal{N}_\Psi(\mathbf{R}^d)}\leq C_2\|\bm x-\bm x'\|^{\alpha/2},
\end{align}
for a constant $C_2$.

In view of (\ref{eq:dist2}), (\ref{eq:entropyC1}) and (\ref{eq:entropyC1X}), there exists a constant $C_3$ such that
\begin{align}\label{eq:dist}
\mathfrak{d}(\bm x, \bm x')^2 \leq C_3 P_{\Phi,\mathbf{X}}\|\bm x-\bm x'\|^{\alpha/4}.
\end{align}
Therefore, the covering number is bounded above by
\begin{align}\label{e1}
\log N(\epsilon,\Omega,\mathfrak{d}) \leq \log N\bigg(\frac{\epsilon^{8/\alpha}}{C_3^{4/\alpha}P_{\Phi,\mathbf{X}}^{4/\alpha}},\Omega,\|\cdot\|\bigg).
\end{align}
The right side of (\ref{e1}) involves the covering number of a Euclidean ball, which is well understood in the literature. See Lemma 4.1 of \cite{pollard1990empirical}. 
This result leads to the bound
\begin{align}\label{eq:entropyBoundC1}
\log N(\epsilon,\Omega,\mathfrak{d})\leq C_{4,0}\log\bigg(\frac{C_{5,0}C_3^{4/\alpha}P_{\Phi,\mathbf{X}}^{4/\alpha}}{\epsilon^{8/\alpha}}\bigg)=:C_4\log\left(\frac{C_5 P^{1/2}_{\Phi,\mathbf{X}}}{\epsilon}\right),
\end{align}
provided that
\begin{align}\label{eq:Rcond1}
\epsilon<C_5P^{1/2}_{\Phi,\mathbf{X}},
\end{align}
where $C_{4,0}$ and $C_{5,0}$ are constants depending on the dimension and the Eucliden diameter of $\Omega$.


\noindent\textbf{Step 2: Bounding the diameter $D$}

Recall that the diameter is defined by $D = \sup_{\bm x, \bm x'\in \Omega}\mathfrak{d}(\bm x, \bm x')$. For any $\bm x,\bm x'\in \Omega$,
\begin{align}\label{eq:distD1}
\mathfrak{d}(\bm x, \bm x')^2 = &\mathbb{E}(g(\bm x)-g(\bm x'))^2 \leq  4\sup _{\bm x\in \Omega}\mathbb{E}(g(\bm x))^2\nonumber\\
                      = & 4\sup _{\bm x\in \Omega}\mathbb{E}(Z(\bm x)-\mathcal{I}_{\Phi,\mathbf{X}}Z(\bm x))^2\nonumber\\
                      = & 4\sup _{\bm x\in \Omega} (\Psi(\bm x-\bm x) - 2\bm r_1^T(\bm x)\mathbf{K}^{-1}r(\bm x) + \bm r^T(\bm x)\mathbf{K}^{-1}\mathbf{K}_1 \mathbf{K}^{-1}\bm r(\bm x)),
\end{align}
where $\bm r$, $\bm r_1$, $\mathbf{K}$ and $\mathbf{K}_1$ are defined in the beginning of Appendix \ref{App:proof}.

Combining identity (\ref{eq:identity}) with
\begin{align*}
\Psi(\bm x_j-\bm x) = \frac{1}{(2\pi)^d}\int_{\mathbb{R}^d}e^{i (\bm x-\bm x_j)^T\bm \omega }\tilde\Psi(\bm\omega)d\bm \omega,
\end{align*}
for any $\bm u = (u_1,...,u_n)$, under Condition \ref{C1}, we have
\begin{align}\label{eq:FrourierIdentity}
& \bm u^T \mathbf{K}_1 \bm u -2\bm u^T \bm r_1(\bm x) + \Psi(\bm x-\bm x)\nonumber\\
= & \frac{1}{(2\pi)^d}\int_{\mathbb{R}^d}\bigg|\sum_{j=1}^n u_j e^{i \bm x_j^T\bm\omega} - e^{i\bm x^T\bm\omega}\bigg|^2 \tilde\Psi(\bm\omega)d\bm\omega\nonumber\\
\leq & \frac{A_1^2}{(2\pi)^d}\int_{\mathbb{R}^d}\bigg|\sum_{j=1}^n u_j e^{i \bm x_j^T\bm\omega} - e^{i\bm x^T\bm\omega}\bigg|^2 \tilde\Phi(\bm\omega)d\bm\omega\nonumber\\
= & A_1^2(\bm u^T \mathbf{K} \bm u -2\bm u^T \bm r(\bm x) + \Phi(\bm x-\bm x)).
\end{align}
We can combine (\ref{eq:FrourierIdentity}) with (\ref{eq:distD1}) by
substituting $\bm u$ in (\ref{eq:FrourierIdentity}) by $\mathbf{K}^{-1}\bm r(\bm x)$ and arrive at
\begin{align*}
\mathfrak{d}(\bm x, \bm x')^2 \leq & 4A_1^2\sup_{\bm x \in \Omega}( \Phi(\bm x-\bm x) - \bm r(\bm x)\mathbf{K}^{-1}\bm r(\bm x)).
\end{align*}
Note that the upper bound of $\Phi(\bm x-\bm x) - \bm r(\bm x)\mathbf{K}^{-1}\bm r(\bm x)$ is $P_{\Phi,\mathbf{X}}^2$, which implies $\mathfrak{d}(\bm x, \bm x')^2 \leq 4A_1^2 P_{\Phi,\mathbf{X}}^2$.
Thus we conclude that
\begin{align}\label{eq:diamC1}
D\leq 2A_1 P_{\Phi,\mathbf{X}}.
\end{align}
\noindent\textbf{Step 3: Bounding the entropy integral}

Under Condition \ref{C1}, if $$P_{\Phi,\mathbf{X}}<C_5^2/A_1^2:=C,$$ (\ref{eq:Rcond1}) is satisfied for all $\epsilon\in [0,D/2]$. Thus, by (\ref{eq:entropyBoundC1}) and (\ref{eq:diamC1}),
\begin{align}\label{eq:boundEC1}
\int_0^{D/2} \sqrt{\log N(\epsilon,\Omega,\mathfrak{d})}d\epsilon & \leq \int_0^{A_1 P_{\Phi,\mathbf{X}}} \sqrt{C_4\log\bigg(  \frac{C_5P_{\Phi,\mathbf{X}}^{1/2}}{\epsilon}\bigg)}d\epsilon\nonumber\\
& \leq \left(\int_0^{A_1 P_{\Phi,\mathbf{X}}}d\epsilon \right)^{1/2}\left(\int_0^{A_1 P_{\Phi,\mathbf{X}}}C_4 \log \left(\frac{C_5 P_{\Phi,\mathbf{X}}}{\epsilon}d\epsilon\right)\right)^{1/2}\\
& =C_{4}^{1/2} A_1 P_{\Phi,\mathbf{X}}\sqrt{\log\left(\frac{C_5 e}{A_ 1 P_{\Phi,\mathbf{X}}^{1/2}}\right)}.
\end{align}
Because $P_{\Phi,\mathbf{X}}\leq 1$, the quantity inside the logrithm can be replaced by $e/P_{\Phi,\mathbf{X}}$ at the cost of (possibly) increasing the constant $C_4$.


\noindent\textbf{Step 4: Bounding $\mathbb{P}(\sup_{\bm x\in\Omega} |Z(\bm x)-\mathcal{I}_{\Phi,\mathbf{X}}Z(\bm x)| > K\int_0^{D/2} \sqrt{\log N(\epsilon,T,\mathfrak{d})}d\epsilon + u)$}

Noting that $\sup_{\bm x\in \Omega}\mathbb{E}(Z(\bm x)-\mathcal{I}_{\Phi,\mathbf{X}}Z(\bm x))^2 = D^2$, by plugging (\ref{eq:diamC1}) into (1.8) in the supplementary materials, we obtain the desired inequality, which completes the proof.

\section{Proof of Theorem \ref{Th:BLUP}}\label{app:B}
	
Denote $\mathbf{Z}=(Z(\bm x_1),\ldots,Z(\bm x_n))^T$.
	Direct calculations show that	
	\begin{align}
	& Y(\bm x)-\hat{Y}_{\mathbf{BLUP},\Phi}(\bm x)\nonumber\\
	= &\underbrace{Z(\bm x)-\bm r^T(\bm x)\mathbf{K}^{-1}\mathbf{Z}}_{I_1}
	-\underbrace{(\bm f^T(\bm x)-\bm r^T(\bm x)\mathbf{K}^{-1}\mathbf{F})(\mathbf{F}^T\mathbf{K}^{-1}\mathbf{F})^{-1}\mathbf{F}^T\mathbf{K}^{-1}\mathbf{Z}}_{I_2}.
	\end{align}
	Thus
	$\sup_{\bm x\in \Omega}|Y(\bm x)-\hat{Y}_{\mathbf{BLUP},\Phi}(\bm x)|\leq \sup_{\bm x\in \Omega}|I_1|+\sup_{\bm x\in \Omega}|I_2|.$
	Clearly, $\sup_{\bm x\in \Omega}|I_1|$ is the uniform error of simple kriging, which is studied in Section \ref{sec:bounds}. Corollary \ref{Coro:1} suggests that $\mathbb{E}\sup_{\bm x\in \Omega}|I_1|=O(P_{\Phi,\mathbf{X}}\log^{1/2}(1/P_{\Phi,\mathbf{X}}))$.
	
Now we turn to $I_2$. By Cauchy-Schwarz inequality,
\begin{align}
|I_2|= & |(\bm f^T(\bm x)-\bm r^T(\bm x)\mathbf{K}^{-1}\mathbf{F})(\mathbf{F}^T\mathbf{K}^{-1}\mathbf{F})^{-1}\mathbf{F}^T\mathbf{K}^{-1}\mathbf{Z}|\nonumber\\
\leq&\left\{(\bm f^T(\bm x)-\bm r^T(\bm x)\mathbf{K}^{-1}\mathbf{F})(\bm f^T(\bm x)-\bm r^T(\bm x)\mathbf{K}^{-1}\mathbf{F})^T\right\}^{1/2}\nonumber\\
&\cdot\left\{\mathbf{Z}^T\mathbf{K}^{-1}\mathbf{F}(\mathbf{F}^T\mathbf{K}^{-1}\mathbf{F})^{-1}(\mathbf{F}^T\mathbf{K}^{-1}\mathbf{F})^{-1}\mathbf{F}^T\mathbf{K}^{-1}\mathbf{Z}\right\}^{1/2}.\label{decomposition}
\end{align}
Note that the right-hand side of (\ref{decomposition}) is the product of a deterministic function and a random variable independent of $\bm x$.

Clearly, the $j$th entry of $\bm f^T(\bm x)-\bm r^T(\bm x)\mathbf{K}^{-1}\mathbf{F}$ is $f_j(\bm x)-\mathcal{I}_{\Phi,X}f_j(\bm x)$, whose absolute value is bounded above by $P_{\Phi,\mathbf{X}}\|f_j\|_{\mathcal{N}_\Phi(\Omega)}$. See Theorem 11.4 of \cite{wendland2004scattered}, also see (1.3) in the supplementary materials. Therefore,
$$(\bm f^T(\bm x)-\bm r^T(\bm x)\mathbf{K}^{-1}\mathbf{F})(\bm f^T(\bm x)-\bm r^T(\bm x)\mathbf{K}^{-1}\mathbf{F})^T\leq P^2_{\Phi,\mathbf{X}}\sum_{j=j}^p\|f_j\|^2_{\mathcal{N}_\Phi(\Omega)} $$

Our final goal is to bound
\begin{align*}
	&\left(\mathbb{E}\left\{\mathbf{Z}^T\mathbf{K}^{-1}\mathbf{F}(\mathbf{F}^T\mathbf{K}^{-1}\mathbf{F})^{-1}(\mathbf{F}^T\mathbf{K}^{-1}\mathbf{F})^{-1}\mathbf{F}^T\mathbf{K}^{-1}\mathbf{Z}\right\}^{1/2}\right)^2\\
	\leq &\mathbb{E}\mathbf{Z}^T\mathbf{K}^{-1}\mathbf{F}(\mathbf{F}^T\mathbf{K}^{-1}\mathbf{F})^{-1}(\mathbf{F}^T\mathbf{K}^{-1}\mathbf{F})^{-1}\mathbf{F}^T\mathbf{K}^{-1}\mathbf{Z}\\
	=&\mathbb{E}\mathbf{Tr}[(\mathbf{F}^T\mathbf{K}^{-1}\mathbf{F})^{-1}\mathbf{F}^T\mathbf{K}^{-1} \mathbf{Z}\mathbf{Z}^T \mathbf{K}^{-1}\mathbf{F}(\mathbf{F}^T\mathbf{K}^{-1}\mathbf{F})^{-1}]\\
	=&\mathbf{Tr}[(\mathbf{F}^T\mathbf{K}^{-1}\mathbf{F})^{-1}\mathbf{F}^T\mathbf{K}^{-1} \mathbf{K}_1 \mathbf{K}^{-1}\mathbf{F}(\mathbf{F}^T\mathbf{K}^{-1}\mathbf{F})^{-1}],
\end{align*}
where $\mathbf{K}_1=(\Psi(\bm x_j-\bm x_k))_{jk}$ is the true correlation.
Via the treatment used in (\ref{eq:identity})-(\ref{eq:h1norm2}), it can be shown that
\begin{eqnarray}\label{cancelK}
\bm \alpha^T \mathbf{K}^{-1}\mathbf{K}_1\mathbf{K}^{-1}\bm \alpha\leq C \bm \alpha^T\mathbf{K}^{-1}\bm \alpha,
\end{eqnarray}
for any $\bm \alpha$ and a constant $C$ depending only on $\|\tilde{\Psi}/\tilde{\Phi}\|_{L_\infty(\mathbf{R}^d)}$, which implies
\begin{align*}
	&\mathbf{Tr}[(\mathbf{F}^T\mathbf{K}^{-1}\mathbf{F})^{-1}\mathbf{F}^T\mathbf{K}^{-1} \mathbf{K}_1 \mathbf{K}^{-1}\mathbf{F}(\mathbf{F}^T\mathbf{K}^{-1}\mathbf{F})^{-1}]\\
	\leq& C \mathbf{Tr}[(\mathbf{F}^T\mathbf{K}^{-1}\mathbf{F})^{-1}]\leq C p/\lambda_{min}(\mathbf{F}^T\mathbf{K}^{-1}\mathbf{F}).
\end{align*}
For $\bm \alpha=(\alpha_1,\ldots,\alpha_p)^T$, we have
\begin{align*}
\lambda_{min}(\mathbf{F}^T\mathbf{K}^{-1}\mathbf{F})=\min_{\|\bm \alpha\|=1}\bm \alpha^T\mathbf{F}^T\mathbf{K}^{-1}\mathbf{F}\bm \alpha=\min_{\|\bm \alpha\|=1}\left\|\mathcal{I}_{\Phi,\mathbf{X}}\sum_{j=1}^p \alpha_j f_j(\bm x)\right\|^2_{\mathcal{N}_\Phi(\Omega)}.
\end{align*}
	Now take a $p$-point subset of $\mathbf{X}$, denoted by $\mathbf{X}_p=\{\bm x'_1,\ldots,\bm x'_p\}$. Define $\mathbf{K}_p=(\Phi(\bm x'_j-\bm x'_k ))_{jk}$ and $\mathbf{F}_p=(\bm f(\bm x'_1),\ldots, \bm f(\bm x'_p))$. Then by (1.5) in the supplementary materials, we have
	\begin{align}
	&\lambda_{min}(\mathbf{F}^T\mathbf{K}^{-1}\mathbf{F})\geq \min_{\|\bm \alpha\|=1}\left\|\mathcal{I}_{\Phi,\mathbf{X}_p}\sum_{j=1}^p \alpha_j f_j(\bm x)\right\|^2_{\mathcal{N}_\Phi(\Omega)}\nonumber\\
	=&\min_{\|\bm \alpha\|=1}\bm \alpha^T\mathbf{F}^T_p\mathbf{K}^{-1}_p\mathbf{F}_p\bm \alpha
	\geq \min_{\bm \alpha\neq 0}\frac{\bm \alpha^T\mathbf{F}^T_p\mathbf{K}^{-1}_p\mathbf{F}_p\bm \alpha}{\alpha^T\mathbf{F}^T_p\mathbf{F}_p\bm \alpha}
	\min_{\bm \alpha\neq 0}\frac{\alpha^T\mathbf{F}^T_p\mathbf{F}_p\bm \alpha}{\bm \alpha^T\bm \alpha}	\nonumber\\
	\geq&\lambda_{min}(\mathbf{K}^{-1}_p)\lambda_{min}(\mathbf{F}^T_p\mathbf{F}_p)\geq \lambda_{min}(\mathbf{F}^T_p\mathbf{F}_p)/\mathbf{Tr}(\mathbf{K}_p)=\lambda_{min}(\mathbf{F}^T_p\mathbf{F}_p)/p.\label{eigenlower}
	\end{align}
	Because $\mathbf{X}'_p$ can be chosen as an arbitrary $p$-point subset, the right-hand side of (\ref{eigenlower}) can be replaced by the maximum value over all possible choices of $\mathbf{F}_p$, which completes the proof.
	
	\section{Proof of Theorem \ref{Th:inconsistency}}
	We use the notation in the proof of Theorem \ref{Th:BLUP}. It is easily verified that $\hat{\beta}-\beta=(\mathbf{F}^T\mathbf{K}^{-1}\mathbf{F})^{-1}\mathbf{F}^T\mathbf{K}^{-1}\mathbf{Z}$. Because $\Phi=\Psi$, $\text{Var}(Z)=\mathbf{K}$. Therefore
	$$\text{Var}(\hat{\bm \beta}-\bm \beta)=(\mathbf{F}^T\mathbf{K}^{-1}\mathbf{F})^{-1}. $$
	Let $\bm \alpha=(\alpha_1,\ldots,\alpha_p)^T$ be an arbitrary vector. Then
	\begin{eqnarray*}
	\bm \alpha^T \mathbf{F}^T\mathbf{K}^{-1}\mathbf{F} \bm \alpha =\left\|\mathcal{I}_{\Phi,\mathbf{X}} \sum_{j=1}^p \alpha_j f_j\right\|_{\mathcal{N}_\Phi(\Omega)}^2\leq \left\| \sum_{j=1}^p \alpha_j f_j\right\|_{\mathcal{N}_\Phi(\Omega)}^2=\bm \alpha^T \mathbf{V}\bm \alpha,
	\end{eqnarray*}
	where the inequality follows from Corollary 10.25 of \cite{wendland2004scattered}; also see (1.4) in the supplementary materials. Clearly, $\mathbf{V}$ is positive definite, because $f_j$'s are linearly independent. Then then desired result follows from the fact that if $\mathbf{A}\leq \mathbf{B}$ then $\mathbf{A}^{-1}\geq \mathbf{B}^{-1}$.

\section{Proof of Theorem \ref{Th:randomkernel}}\label{app:D}

First we prove the simple kriging version. Similar to the proof of Theorem \ref{Th:main}, we examine the distance defined by
\begin{eqnarray*}
\mathfrak{d}^2((\bm x_1,\bm \theta_1),(\bm x_2,\bm \theta_2))=\mathbb{E}\left[Z(\bm x_1)-\mathcal{I}_{\Phi_{\bm \theta_1},\mathbf{X}}Z(\bm x_1)-Z(\bm x_2)+\mathcal{I}_{\Phi_{\bm \theta_2},\mathbf{X}}Z(\bm x_2)\right]^2.
\end{eqnarray*}
It follows from a similar argument as in Theorem \ref{Th:main} that the diameter of $\mathfrak{d}$ is no more than a multiple of $P_{\mathbf{X}}$. It remains to study the cover number given by $\mathfrak{d}$. First we can separate the effect of $\bm x$ and $\bm \theta$ using the following inequality
\begin{align}\label{d2}
  \mathfrak{d}^2((\bm x_1,\bm \theta_1),(\bm x_2,\bm \theta_2))\leq& 2\mathbb{E}\left[Z(\bm x_1)-\mathcal{I}_{\Phi_{\bm \theta_1},\mathbf{X}}Z(\bm x_1)-Z(\bm x_2)+\mathcal{I}_{\Phi_{\bm \theta_1},\mathbf{X}}Z(\bm x_2)\right]^2\nonumber\\
& + 2\mathbb{E}\left[\mathcal{I}_{\Phi_{\bm \theta_1},\mathbf{X}}Z(\bm x_2)-\mathcal{I}_{\Phi_{\bm \theta_2},\mathbf{X}}Z(\bm x_2)\right]^2
\end{align}
The first term in (\ref{d2}) is studied in the proof of Theorem \ref{Th:main}. It suffices to show that
$$\mathbb{E}\left[\mathcal{I}_{\Phi_{\bm \theta_1},\mathbf{X}}Z(\bm x)-\mathcal{I}_{\Phi_{\bm \theta_2},\mathbf{X}}Z(\bm x)\right]^2\leq C P_{\mathbf{X}}^2\|\theta_1-\theta_2\|^2, $$
for all $x\in\Omega$ and some constant $C$.
Let $\mathbf{K}_{\bm \theta_l}=(\Phi_{\bm \theta_l}(\bm x_j-\bm x_k))_{jk}$, and $\bm r_{\bm \theta_l}=(\Phi_{\bm \theta_l}(\bm x-\bm x_1),\ldots,\Phi_{\bm \theta_l}(\bm x-\bm x_n))^T$, for $l=1,2$. By the mean value theorem, we have
\begin{align*}
&\left|\mathcal{I}_{\Phi_{\bm \theta_1},\mathbf{X}}Z(\bm x)-\mathcal{I}_{\Phi_{\bm \theta_2},\mathbf{X}}Z(\bm x)\right|=\left|\mathbf{Z}^T(\mathbf{K}^{-1}_{\bm \theta_1}\bm r_{\bm \theta_1}-\mathbf{K}^{-1}_{\bm \theta_2}\bm r_{\bm \theta_2})\right|\\
\leq&\max_{\bm \theta\in\Theta} \left\|\mathbf{Z}^T\frac{\partial}{\partial \bm \theta}(\mathbf{K}^{-1}_{\bm \theta}\bm r_{\bm \theta}) \right\|\|\bm\theta_1-\bm \theta_2\|.
\end{align*}
It remains to prove that $\mathbb{E}[\mathbf{Z}^T\partial (\mathbf{K}^{-1}_{\bm \theta}\bm r_{\bm \theta})/\partial \theta_l]^2\leq C P_{\mathbf{X}}^2,$ for all $\bm \theta\in\Theta$ and $l=1,\ldots,q$.
For notational simplicity, we denote $\mathbf{K}:=\mathbf{K}_{\bm \theta}, \bm r:=\bm r_{\bm \theta}, \dot{\mathbf{K}}=\partial \mathbf{K}/\partial \theta_l, \dot{\bm r}=\partial \bm r/\partial \theta_l, \dot{\Phi}=\partial \Phi/\partial \theta_l$ and $\dot{\tilde{\Phi}}=\partial \tilde{\Phi}/\partial \theta_l$. As before, denote the covariance matrix of $\mathbf{Z}$ by $\mathbf{K}_1$. Then
\begin{align*}
\mathbb{E}[\mathbf{Z}^T\partial (\mathbf{K}^{-1}\bm r)/\partial \theta_l]^2=&(\bm r^T\mathbf{K}^{-1}\dot{\mathbf{K}}\mathbf{K}^{-1}- \dot{\bm r}^T\mathbf{K}^{-1})\mathbf{K}_1(\mathbf{K}^{-1}\dot{\mathbf{K}}\mathbf{K}^{-1}\bm r-\mathbf{K}^{-1} \dot{\bm r})\\
\leq& C_1 (\bm r^T\mathbf{K}^{-1}\dot{\mathbf{K}}- \dot{\bm r}^T)\mathbf{K}^{-1}(\dot{\mathbf{K}}\mathbf{K}^{-1}\bm r- \dot{\bm r}),
\end{align*}
where the inequality follows from (\ref{cancelK}).

Define $\bm u=(u_1,\ldots,u_n)^T:=\mathbf{K}^{-1}\bm r$, and $h(\bm y):=\sum_{j=1}^n u_j \dot{\Phi}(\bm y-\bm x_j)-\dot{\Phi}(\bm y-\bm x)$. Clearly, the $j$th entry of $\dot{\mathbf{K}}\mathbf{K}^{-1}\bm r-\dot{\bm r}$ is $h(\bm x_j)$. Thus
$$(\bm r^T\mathbf{K}^{-1}\dot{\mathbf{K}}-\dot{\bm r}^T)\mathbf{K}^{-1}(\dot{\mathbf{K}}\mathbf{K}^{-1}\bm r- \dot{\bm r})=\|\mathcal{I}_{\Phi,\mathbf{X}}h\|^2_{\mathcal{N}_\Phi(\mathbf{R}^d)}\leq \|h\|^2_{\mathcal{N}_\Phi(\mathbf{R}^d)}. $$
Finally, we use Fourier transform to calculate $\|h\|^2_{\mathcal{N}_\Phi(\mathbf{R}^d)}$. It is worth noting that the Fourier transform is performed with respect to $\bm y$, not $\bm x$. It is easy to find that $\tilde{h}(\omega)=(\sum_{j=1}^n u_je^{-i\omega \bm x_j}-e^{-i\omega \bm x})\dot{\tilde{\Phi}}(\omega)$, which implies
\begin{align}
\|h\|^2_{\mathcal{N}_\Phi(\mathbf{R}^d)}=&\int_{\mathbf{R}^d} \frac{\tilde{h}(\omega)\bar{\tilde{h}}(\omega)}{\tilde{\Phi}(\omega)}d \omega=\int_{\mathbf{R}^d}\left|\sum_{j=1}^n u_j e^{i\omega \bm x_j}-e^{i\omega \bm x}\right|^2\frac{\dot{\tilde{\Phi}}^2(\omega)}{\tilde{\Phi}(\omega)}d\omega \\
\leq& A_2\int_{\mathbf{R}^d}\left|\sum_{j=1}^n u_j e^{i\omega \bm x_j}-e^{i\omega \bm x}\right|^2\tilde{\Phi}(\omega)d\omega=C_2P_{\Phi,\mathbf{X}}^2\leq C_2P_{\mathbf{X}}^2,
\end{align}
where the first inequality follows from the condition that $|\partial \log \tilde{\Phi}/\partial \bm \theta_l|=|\dot{\tilde{\Phi}}/\tilde{\Phi}|\leq A_2$.

The proof for the universal kriging case follows from similar lines as that of Theorem \ref{Th:BLUP}. Hence we complete the proof.

\end{appendix}

\newpage
  \begin{center}
    {\LARGE\bf  Supplement to ``On Prediction Properties of Kriging: Uniform Error Bounds and Robustness''}
\end{center}

\setcounter{section}{0}
\renewcommand{\thesection}{\arabic{section}}

\section{Auxiliary tools}\label{App:tools}

In this section, we review some mathematical tools which are used in the proofs presented in Appendix.

\subsection{Reproducing kernel Hilbert spaces}\label{App:RKHS}

In this subsection we introduce the reproducing kernel Hilbert spaces and several results from literature. Let $\Omega$ be a subset of $\mathbf{R}^d$. Assume that $K:\Omega \times \Omega \rightarrow \mathbf{R}$ is a symmetric positive definite kernel. Define the linear space
\begin{eqnarray}\label{FPhi}
F_{K}(\Omega)=\left\{\sum_{i=1}^n\beta_i K(\cdot,{x}_i):\beta_i\in \mathbf{R},{x}_i\in \Omega,n\in\mathbb{N}\right\},
\end{eqnarray}
and equip this space with the bilinear form
\begin{eqnarray}
\left\langle\sum_{i=1}^n\beta_i K(\cdot,{x}_i),\sum_{j=1}^m\gamma_j K(\cdot, x'_j)\right\rangle_K:=\sum_{i=1}^n\sum_{j=1}^m\beta_i\gamma_j K({x}_i, x'_j).
\label{eq1.2}
\end{eqnarray}
Then the \emph{reproducing kernel Hilbert space} $\mathcal{N}_{K}(\Omega)$ generated by the kernel function $K$ is defined as the closure of $F_{K}(\Omega)$ under the inner product $\langle\cdot,\cdot\rangle_{K}$, and the norm of  $\mathcal{N}_{K}(\Omega)$  is $\| f\|_{\mathcal{N}_{K}(\Omega)}=\sqrt{\langle f,f\rangle_{\mathcal{N}_{K}(\Omega)}}$, where $\langle\cdot,\cdot\rangle_{\mathcal{N}_{K}(\Omega)}$ is induced by $\langle \cdot,\cdot\rangle_{K}$. More detail about reproducing kernel Hilbert space can be found in \cite{wendland2004scattered} and \cite{wahba1990spline}. In particular, we have the following theorem, which gives another characterization of the reproducing kernel Hilbert space when $K$ is defined by a stationary kernel function $\Phi$, via the Fourier transform of $\Phi$.



\begin{theorem}[Theorem 10.12 of \cite{wendland2004scattered}]\label{thm:NativeSpace}
Let $\Phi$ be a positive definite kernel function which is continuous and integrable in $\mathbf{R}^d$. Define
	$$\mathcal{G}:=\{f\in L_2(\mathbf{R}^d)\cap C(\mathbf{R}^d):\tilde{f}/\sqrt{\tilde{\Phi}}\in L_2(\mathbf{R}^d)\},$$
	with the inner product
	$$\langle f,g\rangle_{\mathcal{N}_\Phi(\mathbf{R}^d)}=(2\pi)^{-d}\int_{\mathbf{R}^d}\frac{\tilde{f}(\bm \omega)\overline{\tilde{g}(\bm \omega)}}{\tilde{\Phi}(\bm \omega)}d \bm \omega.$$ Then $\mathcal{G} = \mathcal{N}_\Phi(\mathbf{R}^d)$, and both inner products coincide.
\end{theorem}

For $f\in \mathcal{N}_\Phi(\Omega)$, a pointwise error bound for the radial basis function interpolation is given by (\cite{wendland2004scattered}, Theorem 11.4):
\begin{eqnarray}\label{firstestimate2}
|f(\bm x)-\mathcal{I}_{\Phi,\mathbf{X}}f(\bm x)|\leq P_{\Phi,\mathbf{X}}(\bm x)\|f\|_{\mathcal{N}_\Phi(\Omega)}.
\end{eqnarray}
In addition, it can be shown that the interpolant $\mathcal{I}_{\Phi,\mathbf{X}}f(\bm x)$ satisfies the following properties (Corollary 10.25, \cite{wendland2004scattered}):
\begin{align}\label{RKHSnormsmall}
    \|\mathcal{I}_{\Phi,\mathbf{X}}f(\bm x)\|_{\mathcal{N}_\Phi(\Omega)} \leq \|f\|_{\mathcal{N}_\Phi(\Omega)}.
\end{align}
In addition, if $\mathbf{X}'\subset\mathbf{X}$, \begin{align}\label{RKHSnormsmallsub}
\|\mathcal{I}_{\Phi,\mathbf{X}'}h\|_{\mathcal{N}_\Phi(\Omega)}\leq \|\mathcal{I}_{\Phi,\mathbf{X}}h\|_{\mathcal{N}_\Phi(\Omega)}.
\end{align}


\subsection{A Maximum inequality for Gaussian processes}


The theory of bounding the maximum value of a Gaussian process is well-established in the literature. The main step of finding an upper bound is to calculate the \textit{covering number} of the index space. Here we review the main results. Detailed discussions can be found in \cite{adler2009random}.

Let $Z_t$ be a Gaussian process indexed by $t\in T$. Here $T$ can be an arbitrary set. The Gaussian process $Z_t$ induces a metric on $T$, defined by
\begin{eqnarray}\label{metric}
\mathfrak{d}(t_1,t_2)=\sqrt{\mathbb{E}(Z_{t_1}-Z_{t_2})^2}.
\end{eqnarray}
The $\epsilon$-covering number of the metric space $(T,\mathfrak{d})$, denoted as $N(\epsilon,T,\mathfrak{d})$, is the minimum integer $N$ so that there exist $N$ distinct balls in $(T,\mathfrak{d})$ with radius $\epsilon$, and the union of these balls covers $T$. Let $D$ be the diameter of $T$. The supremum of a Gaussian process is closely tied to a quantity called the \textit{entropy integral}, defined as
\begin{eqnarray}\label{entropy}
\int_0^{D/2} \sqrt{\log N(\epsilon,T,\mathfrak{d})}d\epsilon.
\end{eqnarray}
Lemma \ref{Th:maximum} gives a maximum inequality for Gaussian processes, which is a direct consequence of Theorems 1.3.3 and 2.1.1 of \cite{adler2009random}.

\begin{lemma}\label{Th:maximum}
Let $Z_t$ be a centered separable Gaussian process on a $\mathfrak{d}$-compact $T$, $\mathfrak{d}$ the metric, and $N$ the $\epsilon$-covering number. Then there exists a universal constant $K$ such that for all $u>0$,
\begin{align}\label{BorellTIS}
\mathbb{P}(\sup_{t\in T} |Z_t| > K\int_0^{D/2} \sqrt{\log N(\epsilon,T,\mathfrak{d})}d\epsilon + u)\leq 2 e^{-u^2/2\sigma^2_T},
\end{align}
where $\sigma^2_T = \sup_{t\in T} \mathbb{E}Z_t^2$.
\end{lemma}

\section{Additional figure related to Table 2}

\begin{figure}[h!]
    \centering
    \begin{subfigure}
        \centering
        \includegraphics[height=2.7in]{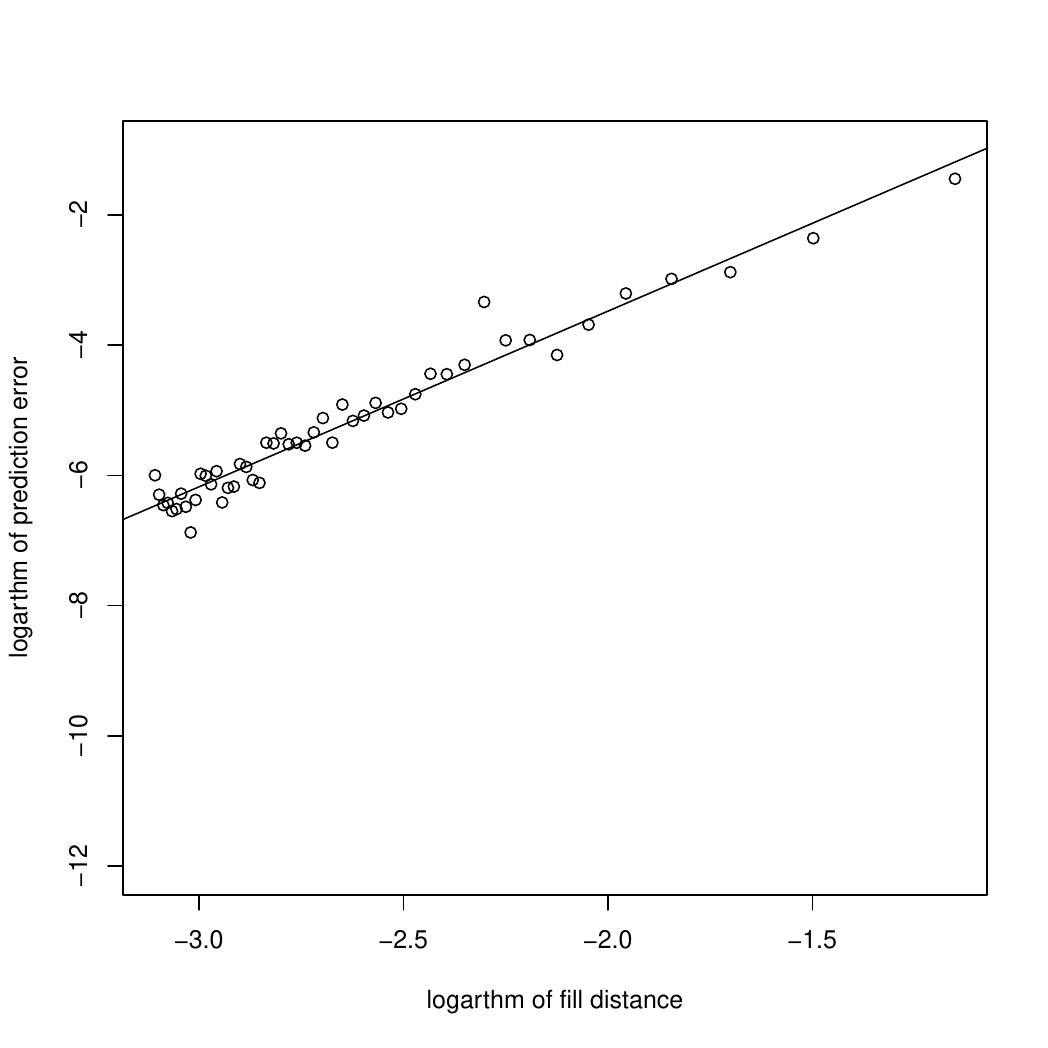}
    \end{subfigure}
    \begin{subfigure}
        \centering
        \includegraphics[height=2.7in]{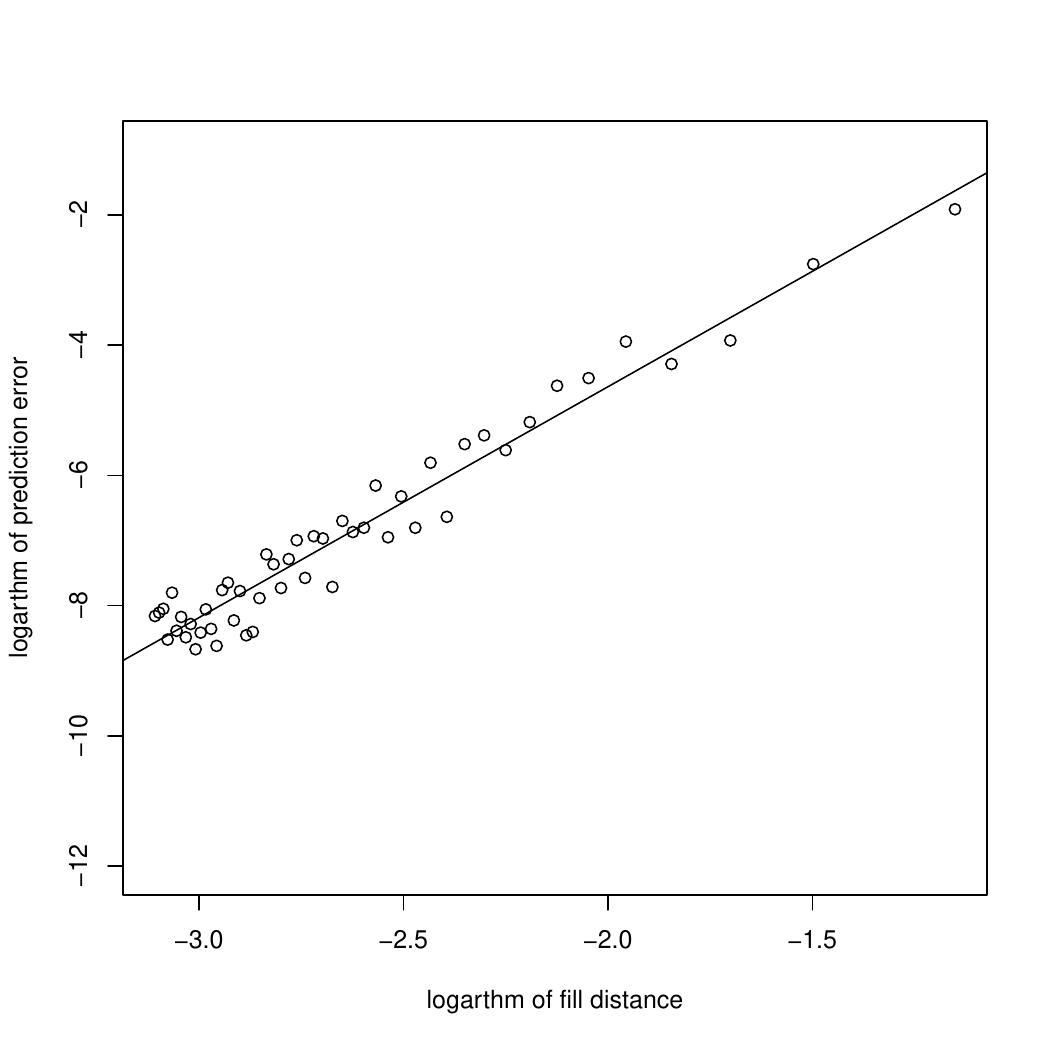}
    \end{subfigure}
    \begin{subfigure}
        \centering
        \includegraphics[height=2.7in]{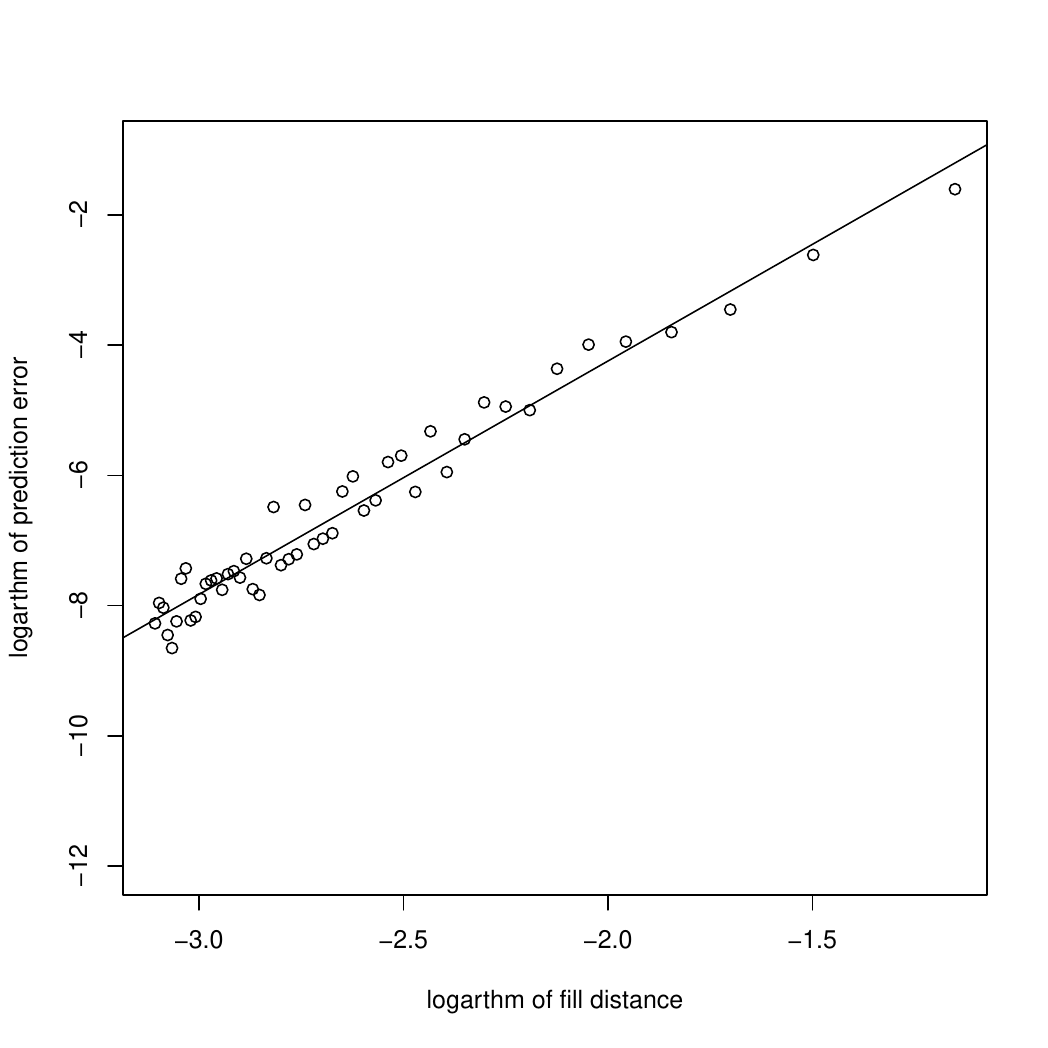}
    \end{subfigure}
    \begin{subfigure}
        \centering
        \includegraphics[height=2.7in]{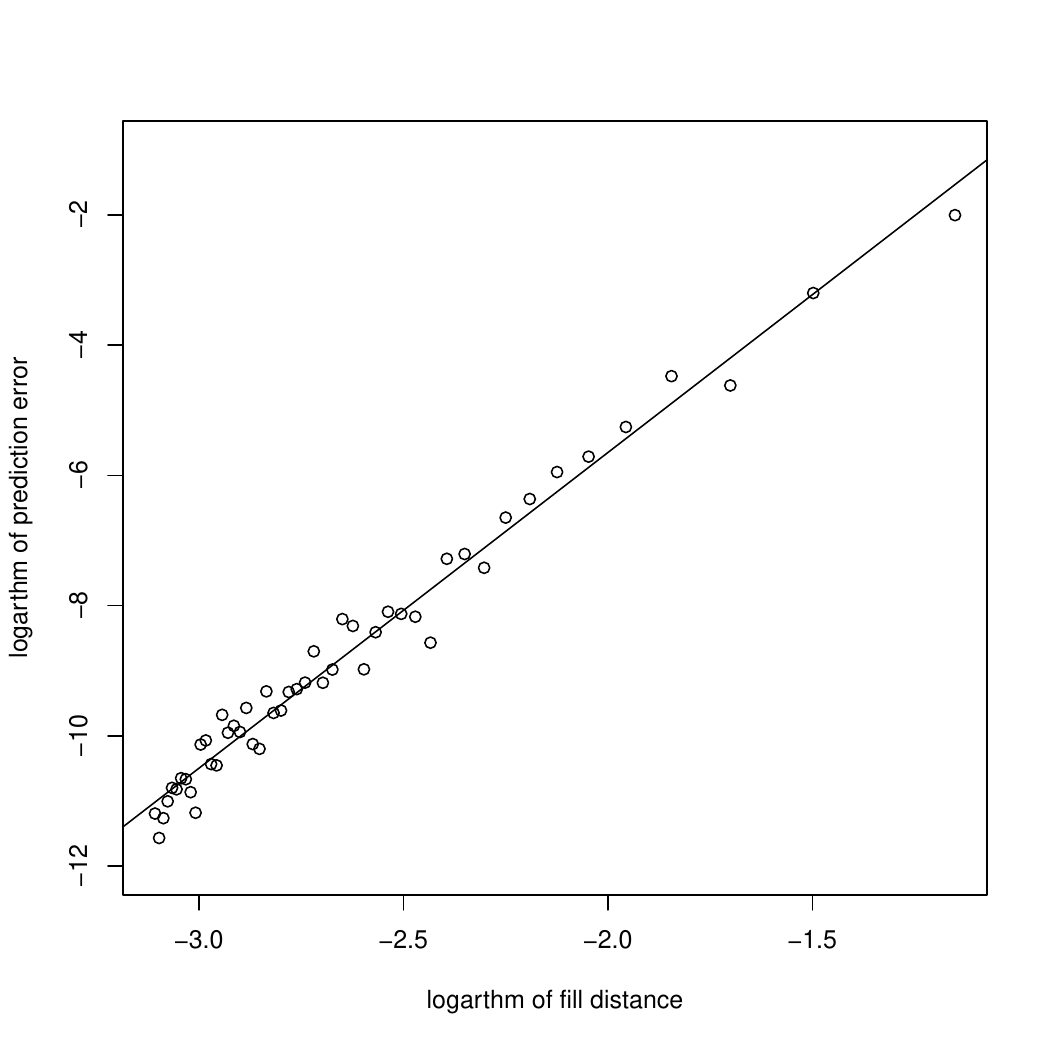}
    \end{subfigure}
    \caption{The regression line of $\log \sup_{\bm x\in \Omega}\epsilon(\bm x)$ on $\log h_{\mathbf{X}}$. Each point denotes one average prediction error for each $n$. {\bf Panel 1:} $\nu_0 = 3$, $\nu = 2.5$. {\bf Panel 2:} $\nu_0 = 5$, $\nu = 3.5$. {\bf Panel 3:} $\nu_0 = \nu = 3.5$. {\bf Panel 4:} $\nu_0 = \nu = 5$.}
    \label{fig:nu0big}
\end{figure}
Figure \ref{fig:nu0big} shows the relationship between the logarithm of the fill distance (i.e., $\log h_{\mathbf{X}}$) and the logarithm of the average prediction error (i.e., $\log \mathcal{E}$) in scatter plots for the four cases given in Table 2. The solid line in each panel shows the linear regression fit calculated from the data. Each of the regression lines in Figure \ref{fig:nu0big} fits the data very well, which gives an empirical confirmation of the approximation in (4.2) in the main text. It is also observed from Figure \ref{fig:nu0big} that, as the fill distance decreases, the maximum prediction error also decreases.

\bibliography{universal}

\bibliographystyle{chicago}

\end{document}